\documentclass[11pt]{amsart}
\pdfoutput=1
\usepackage{heuristica}
\usepackage[heuristica,vvarbb,bigdelims]{newtxmath}

\usepackage{fullpage,url,mathrsfs,amssymb,color,multicol,enumerate,amscd}
\usepackage[alphabetic,backrefs,lite]{amsrefs} 
\usepackage{multirow}
\usepackage{mathtools}
\usepackage[all]{xy} 
\usepackage{verbatim,moreverb}

\usepackage[OT2,T1]{fontenc}

\usepackage{tikz-cd}
\usepackage{xfrac}
\usepackage{float}

\numberwithin{equation}{section}
\usepackage{setspace}
\usepackage{color}
\newcommand{\vanish}[1]{}

\def\bbar#1{\setbox0=\hbox{$#1$}\dimen0=.2\ht0 \kern\dimen0 }
 
\usepackage{enumitem}

\DeclareSymbolFont{cyrletters}{OT2}{wncyr}{m}{n}
\DeclareMathSymbol{\Sha}{\mathalpha}{cyrletters}{"58}

\newcommand{\F}{{\mathbb F}}  
\newcommand{\G}{{\mathbb G}}

\newcommand{\PP}{{\mathbb P}}
\newcommand{\Q}{{\mathbb Q}}

\newcommand{\ZZ}{{\mathbb Z}}
\newcommand{\Zhat}{{\widehat{\ZZ}}}

\def\bbar#1{\setbox0=\hbox{$#1$}\dimen0=.2\ht0 \kern\dimen0 \overline{\kern-\dimen0 #1}}
\newcommand{\Qbar}{{\overline{\mathbb Q}}}

\newcommand{\calF}{{\mathcal F}}

\DeclareMathOperator{\Aut}{Aut}
\DeclareMathOperator{\Gal}{Gal}

\newcommand{\tors}{{\operatorname{tors}}}

\newcommand{\GL}{\operatorname{GL}}
\newcommand{\SL}{\operatorname{SL}}

\newcommand{\injects}{\hookrightarrow}

\newcommand{\intersect}{\cap} 

\newcommand{\union}{\cup}

\usepackage{amsthm}

\definecolor{webcolor}{rgb}{0,0,1}
\definecolor{webbrown}{rgb}{.6,0,0}
\usepackage[
        colorlinks,
        linkcolor=webbrown,  filecolor=webbrown,  citecolor=webbrown, 
        backref,
        pdfauthor={Rakvi}, 
        pdftitle={},
]{hyperref}
\usepackage{cleveref}
\showboxdepth=\maxdimen
\showboxbreadth=\maxdimen
\newtheorem{theorem}{Theorem}[section]
\newtheorem{lemma}[theorem]{Lemma}
\newtheorem{Corollary}[theorem]{Corollary}
\newtheorem{proposition}[theorem]{Proposition}

\theoremstyle{definition}
\newtheorem{definition}[theorem]{Definition}

\newtheorem*{Mazur's Program B}{Mazur's Program B}

\theoremstyle{remark}

\crefname{section}{§}{§§}
\crefname{lemma}{Lemma}{Lemmas}
\crefname{equation}{equation}{equations}
\crefname{theorem}{Theorem}{Theorems}
\crefname{proposition}{Proposition}{Propositions}
\crefname{Corollary}{Corollary}{Corollaries}

\usepackage{listings}
\usepackage{datetime}
\date{\today}
\usepackage{mathtools}
\usepackage{amsfonts}
\usepackage{hyperref}
\usepackage{microtype}
\begin{document}

\title{ On possibilities of 3-adic Galois images associated to isogeny-torsion graphs}    
\author{Rakvi}
\address{Department of Mathematics, The University of Maine, Orono, Maine 04469}
\email{rakvi.lecturer@maine.edu}


\begin{abstract} 
Let $E$ be a non CM elliptic curve defined over $\Q$. There is an isogeny-torsion graph associated to $E$ and there is also a Galois representation $\rho_{E,l^{\infty}} \colon \Gal(\Qbar/\Q) \to \GL_2(\ZZ_{\ell})$ associated to $E$ for every prime $\ell.$ In this article, we explore the relation between these two objects when $\ell=3$. More precisely, we give a classification of 3-adic Galois images associated to vertices of the isogeny-torsion graph of $E.$

\end{abstract}

\maketitle
\setcounter{tocdepth}{1}
\section{Introduction}

\subsection{Galois Representations}
Let $E/\Q$ be an elliptic curve defined over $\Q.$ Fix an algebraic closure $\Qbar$ of $\Q.$ Define $\G_{\Q} \coloneqq \Gal(\Qbar/\Q).$ Let $N \ge 1$ be a natural number. We will denote the set of $N$ torsion points of $E$ by $E[N]$. It is a free module over $\ZZ/N\ZZ$ of rank $2.$ By choosing a basis for $E[N]$ we can identify $\Aut(E[N])$ with $\GL_2(\ZZ/N\ZZ).$ Since $E$ is defined over $\Q$, there is a natural Galois action of $\G_{\Q}$ on $E[N]$ which gives rise to a mod $N$ representation \[ \rho_{E,N} \colon \G_{\Q} \to \Aut(E[N]) \simeq \GL_2(\ZZ/N\ZZ).\] 

For a prime $p$, let $\ZZ_p$ denote the ring of $p$-adic integers. For a composite $N$, define $\ZZ_N \coloneqq \prod_{\ell \mid N} \GL_2(\ZZ_{\ell}).$ By choosing a compatible system of bases for all $N$ bigger than or equal to $1$, we get the adelic representation associated to $E$ 
\[ \rho_{E} \colon \G_{\Q} \to \varprojlim_{N \ge 1} \GL_2(\ZZ/N\ZZ) \simeq \GL_2(\Zhat) \simeq \prod_{\ell} \GL_2(\ZZ_{\ell}).\] We get the $N$-adic representation \[\rho_{E,N^{\infty}} \colon \G_{\Q} \to \GL_2(\ZZ_N)\] by taking the composition of $\rho_E$ with the projection map $\pi_{N^{\infty}} \colon \GL_2(\Zhat) \to \GL_2(\ZZ_N).$ 

While the question of giving a complete classification of adelic images i.e.,  $\rho_E(\G_{\Q})$ for all $E/\Q$ without CM is still open, $\ell$-adic images i.e., $\rho_{E,l^{\infty}}(\G_{\Q})$ for all $E/\Q$ without CM have been determined. In 2015, Rouse--Zureick-Brown \cite{MR3500996} gave a classification of 2-adic images and in 2021, Rouse--Sutherland--Zureick-Brown \cite{MR4468989} gave a classification of $\ell$-adic images for odd primes $\ell.$ In a recent breakthrough article of Zywina \cite{2206.14959} he gives an algorithm to compute $\rho_E(\G_{\Q})$ given an $E/\Q$ without CM. 

\subsection{Isogenies}
Let $E/\Q$ be an elliptic curve defined over $\Q$.  If $E$ is without CM, then it can be shown each elliptic curve that is isogenous to $E$ is also without CM. Let us call the set of all elliptic curves $E'/\Q$ that are isogenous to $E$ taken up to isomorphism over the $\Q$-isogeny class of $E$ and denote it by $\frak{E}.$ The isogeny graph associated to $\frak{E}$ is a connected graph whose vertices are elliptic curves in $\frak{E}$, two vertices are adjacent if and only if there is an isogeny of prime degree between them with the isogeny degree as the label of a connecting edge. Between any two vertices in an isogeny graph there is a unique degree of cyclic isogeny between them. The isogeny-torsion graph associated to $\frak{E}$ is an isogeny graph where vertices are labelled with their torsion group structure over $\Q.$ 
In 2021, Chiloyan and Lozano-Robledo \cite{chiloyan2021classification} gave a classification of isogeny-torsion graphs associated to all elliptic curves defined over $\Q.$ We will borrow their notation to denote these graphs. For example, we will say that an isogeny graph is of type $L_2(p)$ if there are exactly two vertices in $\frak{E}$ and they are connected with an edge of label $p$ for $p \in \{2,3,5,7,11,13,17,19,37,43,67,163\}.$ 

A motivation behind this article is to understand the relationship between Galois images and isogeny. More precisely, we give a classification of 3-adic images that can occur at vertices of isogeny-torsion graphs associated to elliptic curves without CM defined over $\Q.$ In related works, Chiloyan \cite{2302.06094}, \cite{2302.06094} gives a classification of 2-adic images that can occur at vertices of isogeny-torsion graphs associated to elliptic curves defined over $\Q.$

We will use RSZB format of label to denote subgroups of $\GL_2(\Zhat).$ For more details on this label, please refer to section 2.4 of \cite{MR4468989}.
Our main Theorem is stated below.

\begin{theorem} \label{thm:mainthm}
    Let $E$ be an elliptic curve defined over $\Q$ without CM. Then, the possibilities for the 3-adic images that can occur at the vertices of isogeny torsion graph associated to $E$ are as follows:

    \begin{table}[H] 
\caption{Main Table}
\label{Table:maintable} 
\renewcommand{\arraystretch}{1.3}
\vspace{0.5 cm}
\begin{minipage}{\textwidth} 
\renewcommand*\footnoterule{}
\centering
\begin{tabular}{|l|l|}
\hline
Isogeny-torsion graph    & Possibilities for 3-adic images of vertices \\ \hline

$L_1$ &   Table \ref{Table:L_1}   \\ \hline
$L_2(2)$ & Table \ref{Table:L_2(2)} \\ \hline
$L_2(3)$ & Table \ref{Table:L_2(3)} \\ \hline
$L_2(5)$ &  Table \ref{Table:L_2(5)} \\ \hline
$L_2(p)$ for $p \in \{7,11,13,17,37\}$       &  Each vertex has $\GL_2(\ZZ_3)$ as 3-adic image \\ \hline
$L_3(9)$ & Table \ref{Table:L_3(9)} \\ \hline
$L_3(25)$ & Each vertex has $\GL_2(\ZZ_3)$ as 3-adic image \\ \hline
$R_4(6)$ & Table \ref{Table:R_4(6)} \\ \hline
$R_4(10)$ & Each vertex has $\GL_2(\ZZ_3)$ as 3-adic image \\ \hline
$R_4(15)$ & Table \ref{Table:R_4(15)} \\ \hline
$R_4(21)$ & Table \ref{Table:R_4(21)} \\ \hline
$R_6$ & Table \ref{Table:R_6} \\ \hline
$T_k$ for $k \in \{4,6,8\}$ & Each vertex has $\GL_2(\ZZ_3)$ as 3-adic image \\ \hline
$S$ & Table \ref{Table:S} \\ \hline

    \end{tabular}
    \end{minipage}
\end{table}
\end{theorem}
 \begin{table}[H] 
\caption{Table for $L_1$}
\label{Table:L_1} 
\renewcommand{\arraystretch}{1.3}
\vspace{0.5 cm}
\begin{minipage}{\textwidth} 
\renewcommand*\footnoterule{}
\centering
\begin{tabular}{|l|l|l|l|}
\hline
Type    & Torsion arrangement & 3-adic images at the vertices    & Example \\ \hline

$L_1$ & \texttt{Trivial}   & $\GL_2(\ZZ_3)$ &   \texttt{37.a}     \\ \hline
    &           &  3.3.0.1     &   \texttt{245.a}     \\ \hline
    &           &  3.6.0.1     &   \texttt{1210.d}     \\ \hline
    &           &  9.9.0.1     &   \texttt{864.e}     \\ \hline
    &           &  9.18.0.1     &   \texttt{2057.c}     \\ \hline
    &           &  9.18.0.2     &   \texttt{22898.d}     \\ \hline
    &           &  9.27.0.1     &   \texttt{1944.e}     \\ \hline
    &           &  9.27.0.2     &   \texttt{118579.b}     \\ \hline

    \end{tabular}
    \end{minipage}
\end{table}

\begin{table}[H] 
\caption{Table for $L_2(2)$}
\label{Table:L_2(2)} 
\renewcommand{\arraystretch}{1.3}
\vspace{0.5 cm}
\begin{tabular}{|l|l|l|l|}
\hline
Type    & Torsion arrangement & 3-adic images at the vertices    & Example \\ \hline

$L_2(2)$& $(\ZZ/2\ZZ , \ZZ/2\ZZ)$ &  $(\GL_2(\ZZ_3),\GL_2(\ZZ_3))$     &   \texttt{46.a}     \\ \hline
        &                         &   (3.3.0.1,3.3.0.1) & \texttt{1568.a}    \\ \hline
        &                          &  (3.6.0.1,3.6.0.1)  & \texttt{726.b}     \\ \hline

    \end{tabular}
\end{table}

\begin{table}[H]  
\caption{Table for $L_2(3)$}
\label{Table:L_2(3)}
\renewcommand{\arraystretch}{1.3}
\vspace{0.5 cm}
\begin{tabular}{|l|l|l|l|}
\hline
Type    & Torsion arrangement & 3-adic images at the vertices    & Example \\ \hline

$L_2(3)$ &  \texttt{(Trivial,Trivial)} & (3.4.0.1,3.4.0.1) & \texttt{176.a} \\ \hline
         &                    & (9.12.0.2,9.12.0.2) & \texttt{196.a} \\ \hline
         &                    & (9.36.0.7, 9.36.0.9) & \texttt{1734.k} \\ \hline
         &                     & (9.36.0.8,9.36.0.8) & \texttt{17100.r} \\\hline
         &    $(\ZZ/3\ZZ, \texttt{Trivial})$ & (3.8.0.1,3.8.0.2) & \texttt{44.a} \\\hline
         &                      &    (9.24.0.2,9.24.0.4)  & \texttt{196.b}\\\hline
        &                      &    (9.72.0.8,9.72.0.16)  & \texttt{486.d}\\\hline
        &                      & (9.72.0.9,9.72.0.15) &   \texttt{17100.j}     \\ \hline
        &                      &    (9.72.0.10,9.72.0.14)  & \texttt{486.c} \\\hline       

    \end{tabular}
\end{table}

\begin{table}[H] 
\caption{Table for $L_2(5)$}
\label{Table:L_2(5)} 
\renewcommand{\arraystretch}{1.3}
\vspace{0.5 cm}
\begin{tabular}{|l|l|l|l|}
\hline
Type    & Torsion arrangement & 3-adic images at the vertices    & Example \\ \hline

$L_2(5)$ &  \texttt{(Trivial,Trivial)} & $(\GL_2(\ZZ_3),\GL_2(\ZZ_3))$ & \texttt{75.c} \\ \hline
         &                    & (3.3.0.1,3.3.0.1) & \texttt{1369.a} \\ \hline
         &                    & (3.6.0.1, 3.6.0.1) & \texttt{338.b} \\ \hline
         &                     & (9.9.0.1,9.9.0.1) & \texttt{43264.f} \\\hline
         &    $(\ZZ/5\ZZ, \texttt{Trivial})$ & $(\GL_2(\ZZ_3),\GL_2(\ZZ_3))$ & \texttt{38.b} \\\hline

    \end{tabular}
\end{table}

 \begin{table}[H] 
\caption{Table for $L_3(9)$}
\label{Table:L_3(9)} 
\renewcommand{\arraystretch}{1.3}
\vspace{0.5 cm}
\begin{tabular}{|l|l|l|l|}
\hline
Type    & Torsion arrangement & 3-adic images at the vertices    & Example \\ \hline

$L_3(9)$ & \texttt{(Trivial,Trivial,Trivial)}   & (9.12.0.1,3.12.0.1,9.12.0.1) &   \texttt{175.b}     \\ \hline
    &           &  (9.36.0.5,9.36.0.1,9.36.0.4)     &   \texttt{432.b}     \\ \hline
    &           &  (9.36.0.6,9.36.0.3,9.36.0.6)    &   \texttt{22491.u}     \\ \hline
    &           &  (27.36.0.1,9.36.0.2,27.36.0.1)     &   \texttt{304.c}     \\ \hline
    & $(\texttt{Trivial},\ZZ/3\ZZ,\ZZ/3\ZZ)$          &  (9.24.0.3,3.24.0.1,9.24.0.1)     &   \texttt{26.a}     \\ \hline
    &           &  (9.72.0.11,9.72.0.2,9.72.0.6)    &   \texttt{54.a}     \\ \hline
    &           &  (9.72.0.13,9.72.0.4,9.72.0.7)     &   \texttt{3213.k}     \\ \hline
    &           &  (27.72.0.2,9.72.0.3,27.72.0.1)     &   \texttt{19.a}     \\ \hline
    &  $(\texttt{Trivial},\ZZ/3\ZZ,\ZZ/9\ZZ)$          &  (9.72.0.12,9.72.0.1,9.72.0.5)     &  \texttt{54.b}    \\ \hline

\end{tabular}
\end{table}

 \begin{table}[H] 
\caption{Table for $R_4(6)$}
\label{Table:R_4(6)} 
\renewcommand{\arraystretch}{1.3}
\vspace{0.5 cm}
\begin{tabular}{|l|l|l|l|}
\hline
Type    & Torsion arrangement & 3-adic images at the vertices    & Example \\ \hline

$R_4(6)$ & $(\ZZ/2\ZZ,\ZZ/2\ZZ,\ZZ/2\ZZ,\ZZ/2\ZZ)$  & (3.4.0.1,3.4.0.1,3.4.0.1,3.4.0.1) &   \texttt{80.b}     \\ \hline
    & $(\ZZ/2\ZZ,\ZZ/2\ZZ,\ZZ/6\ZZ,\ZZ/6\ZZ)$          &  (3.8.0.2,3.8.0.2,3.8.0.1,3.8.0.1)     &   \texttt{20.a}     \\ \hline
\end{tabular}
\end{table}

\begin{table}[H] 
\caption{Table for $R_4(15)$}
\label{Table:R_4(15)} 
\renewcommand{\arraystretch}{1.3}
\vspace{0.5 cm}
\begin{tabular}{|l|l|l|l|}
\hline
Type    & Torsion arrangement & 3-adic images at the vertices    & Example \\ \hline

$R_4(15)$ & $(\texttt{Trivial},\texttt{Trivial},\texttt{Trivial},\texttt{Trivial})$  & (3.4.0.1,3.4.0.1,3.4.0.1,3.4.0.1) &   \texttt{400.d}     \\ \hline
         & $(\ZZ/5\ZZ,\ZZ/5\ZZ,\texttt{Trivial},\texttt{Trivial})$  & (3.4.0.1,3.4.0.1,3.4.0.1,3.4.0.1) &   \texttt{50.b}     \\ \hline
    & $(\ZZ/3\ZZ,\ZZ/3\ZZ,\texttt{Trivial},\texttt{Trivial})$          &  (3.8.0.1,3.8.0.1,3.8.0.2,3.8.0.2)     &   \texttt{50.a}     \\ \hline
\end{tabular}
\end{table}

\begin{table}[H] 
\caption{Table for $R_4(21)$}
\label{Table:R_4(21)} 
\renewcommand{\arraystretch}{1.3}
\vspace{0.5 cm}
\begin{tabular}{|l|l|l|l|}
\hline
Type    & Torsion arrangement & 3-adic images at the vertices    & Example \\ \hline

$R_4(21)$ & $(\texttt{Trivial},\texttt{Trivial},\texttt{Trivial},\texttt{Trivial})$  & (3.4.0.1,3.4.0.1,3.4.0.1,3.4.0.1) &   \texttt{1296.f}     \\ \hline
        
    & $(\ZZ/3\ZZ,\ZZ/3\ZZ,\texttt{Trivial},\texttt{Trivial})$          &  (3.8.0.1,3.8.0.1,3.8.0.2,3.8.0.2)     &   \texttt{162.b}     \\ \hline
\end{tabular}
\end{table}

\begin{table}[H] 
\caption{Table for $R_6$}
\label{Table:R_6} 
\renewcommand{\arraystretch}{1.3}
\vspace{0.5 cm}
\begin{tabular}{|l|l|l|l|}
\hline
Type    & Torsion arrangement & 3-adic images at the vertices    & Example \\ \hline

$R_6$ & $(\ZZ/2\ZZ,\ZZ/2\ZZ,\ZZ/2\ZZ,\ZZ/2\ZZ,\ZZ/2\ZZ,\ZZ/2\ZZ)$  & \tiny{(9.12.0.1,9.12.0.1,3.12.0.1,3.12.0.1,9.12.0.1,9.12.0.1)} &   \texttt{98.a}     \\ \hline
        
    & $(\ZZ/2\ZZ,\ZZ/2\ZZ,\ZZ/6\ZZ,\ZZ/6\ZZ,\ZZ/6\ZZ,\ZZ/6\ZZ)$          &  \tiny{(9.24.0.3,9.24.0.3,3.24.0.1,3.24.0.1,9.24.0.1,9.24.0.1)}     &   \texttt{14.a}     \\ \hline
\end{tabular}
\end{table}

\begin{table}[H] 
\caption{Table for $S$}
\label{Table:S} 
\renewcommand{\arraystretch}{1.3}
\vspace{0.5 cm}
\begin{tabular}{|l|l|l|l|}
\hline
Type    & Torsion arrangement & 3-adic images at the vertices    & Example \\ \hline

$S$ & \tiny{$(\ZZ/2\ZZ \times \ZZ/2\ZZ,\ZZ/2\ZZ \times \ZZ/2\ZZ,\ZZ/2\ZZ,\ZZ/2\ZZ,\ZZ/2\ZZ,\ZZ/2\ZZ,\ZZ/2\ZZ,\ZZ/2\ZZ)$}  & \tiny{(3.4.0.1,3.4.0.1,3.4.0.1,3.4.0.1,3.4.0.1,3.4.0.1,3.4.0.1,3.4.0.1)} &   \texttt{240.b}     \\ \hline
        
   &\tiny{$(\ZZ/2\ZZ \times \ZZ/2\ZZ,\ZZ/2\ZZ \times \ZZ/2\ZZ,\ZZ/4\ZZ,\ZZ/4\ZZ,\ZZ/2\ZZ,\ZZ/2\ZZ,\ZZ/2\ZZ,\ZZ/2\ZZ)$}  & \tiny{(3.4.0.1,3.4.0.1,3.4.0.1,3.4.0.1,3.4.0.1,3.4.0.1,3.4.0.1,3.4.0.1)} &   \texttt{150.b}     \\ \hline
    &\tiny{$(\ZZ/2\ZZ \times \ZZ/6\ZZ,\ZZ/2\ZZ \times \ZZ/2\ZZ,\ZZ/6\ZZ,\ZZ/2\ZZ,\ZZ/6\ZZ,\ZZ/2\ZZ,\ZZ/6\ZZ,\ZZ/2\ZZ)$}  & \tiny{(3.8.0.1,3.8.0.2,3.8.0.1,3.8.0.2,3.8.0.1,3.8.0.2,3.8.0.1,3.8.0.2)} &   \texttt{30.a}     \\ \hline
     &\tiny{$(\ZZ/2\ZZ \times \ZZ/6\ZZ,\ZZ/2\ZZ \times \ZZ/2\ZZ,\ZZ/12\ZZ,\ZZ/4\ZZ,\ZZ/6\ZZ,\ZZ/2\ZZ,\ZZ/6\ZZ,\ZZ/2\ZZ)$}  & \tiny{(3.8.0.1,3.8.0.2,3.8.0.1,3.8.0.2,3.8.0.1,3.8.0.2,3.8.0.1,3.8.0.2)} &   \texttt{90.c}     \\ \hline
\end{tabular}
\end{table}

We give a brief outline of the structure of this article. In Section \ref{sec:background material}, we discuss some background material related to modular curves. In Section \ref{sec:isogenies}, we discuss results on isogenies of elliptic curves that will be used in later sections. In Section \ref{sec:Analysisofrationalpoints}, we find all rational points on curves of our interest. In Section \ref{sec:results}, we prove Theorem \ref{thm:mainthm}. All the code related to this paper can be found in its github repository \cite{RakviGitHub} which is available at \url{https://github.com/Rakvi6893/Variation-of-three-adic-images-with-isogeny}.
\section{Modular Curves}\label{sec:background material}

In this section, we discuss some background material about modular curves that will be used in later sections.

\subsection{Modular Curves} Let $G$ be an open subgroup of $\GL_2(\Zhat)$ that contains $-I$ and has full determinant, i.e., $\det(G)=\Zhat^{\times}$. The level of $G$ is defined as the smallest positive integer $N$ such that $G$ is equal to $\pi_N^{-1}(\pi_N(G))$, where $\pi_N \colon \GL_2(\Zhat) \to \GL_2(\ZZ/N\ZZ)$ is the natural projection map. Let $N$ be the level of $G$. By working modulo $N$, we can think of $G$ as a subgroup of $\GL_2(\ZZ/N\ZZ)$. 

Associated to $G$, there is a smooth, projective and geometrically connected curve $X_G$ defined over $\Q.$ Modular curves can be defined using two approaches. 

One approach is to define $X_{G}$ as the generic fiber of the coarse space for the stack $M_{G}$ which parametrizes elliptic curves with $G$-level structure. Please refer to Chapter $4$, Section 3 of \cite{MR0337993} for details on this construction. Another approach is to construct the field of meromorphic functions $\calF$ on a Riemann surface which is Galois over $\Q(j)$ (here $j$ is the modular $j$-invariant function) and has an action of $\GL_2(\Zhat)$ associated with it. Then we define $X_G$ as smooth curve whose function field is the fixed field corresponding to $G.$ Please see Theorem $6.6$, Chapter $6$ of \cite{MR1291394} for details on this construction.

If $G \subseteq G'$, then there is a morphism $X_G \to X_{G'}.$ In particular, there is a morphism $\pi_G \colon X_G \to \PP^1_{\Q}$ associated to every $X_G.$ We have $\pi_G(X_G(\Q)) \subseteq \Q \union {\infty}.$ The points that lie above $\infty$ are known as cusps. For non cuspidal points the following parametrization property holds. 

\begin{lemma}\label{lemma:modular curves}
    For an elliptic curve $E$ defined over $\Q$ such that the $j$-invariant of $E$ is not equal to $0$ or $1728$ we have $\rho_E(\G_{\Q})$ is conjugate to a subgroup of $G$ if and only if the $j$-invariant of $E$ lies in $\pi_G(X_G(\Q)).$
\end{lemma}

\begin{proof}
    See Proposition 3.3 of \cite{1508.07660} for a proof.
\end{proof}

When $G$ is Borel subgroup of $\GL_2(\ZZ/N\ZZ)$, i.e., it consists of upper triangular matrices modulo $N$, we denote $X_G$ by $X_0(N).$ Rational non cuspidal points of $X_0(N)$ correspond to elliptic curves defined over $\Q$ up to isomorphism that admit a cyclic isogeny of degree $N.$ 

When \[G=\left\{\begin{psmallmatrix}\pm 1 & a \\0 & b\end{psmallmatrix} \colon a \in \ZZ/N\ZZ, b \in (\ZZ/N\ZZ)^{\times}\right\},\]
we denote $X_G$ by $X_{\pm 1}(N).$ Rational non cuspidal points of $X_{\pm 1}(N)$ correspond to elliptic curves defined over $\Q$ up to isomorphism that have a point $P$ of order $N$ such that $\pm P$ is defined over $\Q.$

For our application we need to analyze if there are non CM elliptic curves whose adelic Galois image is simultaneously contained in two open subgroups of $\GL_2(\Zhat)$ that contain $-I$ and have full determinant. Therefore, we now discuss fiber product of two modular curves that we will study in later sections. If $G_1$ and $G_2$ are two open subgroups of $\GL_2(\Zhat)$ of level $N_1$ and $N_2$ respectively that contain $-I$ and have full determinant, then the fiber product of modular curves $X_{G_1}$ and $X_{G_2}$ is their fiber product as stacks over $X_{\GL_2(\Zhat)}.$ If $N_1$ and $N_2$ are coprime, then the fiber product of $X_{G_1}$ and $X_{G_2}$ is the modular curve $X_G$ where $G$ is $G_1 \intersect G_2.$

For each fiber product that we construct using generators of corresponding factor curves, we compute its canonical model using the function \texttt{FindModelOfXG} available at \href{https://github.com/davidzywina/OpenImage/blob/master/main/ModularCurves.m}{github repository} which is part of the algorithm given by Zywina \cite{2206.14959} to compute $\rho_E(\G_{\Q})$ for a given non CM elliptic curve $E$ defined over $\Q.$ Please see Section $5$ of \cite{2206.14959} for details on computing canonical models of modular curves. We then analyse the rational points on it.

\subsection{Cusps} Fix a modular curve $(X_G,\pi_G).$ Let $N$ be the level of $G.$ Recall that cusps are points in $X_G(\mathbb{C})$ that lie above $\infty$ via $\pi_G.$ In this section we discuss a method to compute the number of cusps using $G.$

Let $U$ be the group of upper triangular matrices in $\SL_2(\ZZ).$ Define $SC:=G\textbackslash \GL_2(\ZZ/N\ZZ)/U_N$, where $U_N$ is the image of $U$ modulo $N.$

\begin{lemma}
    We can identify $SC$ with the set of cusps of $X_G(\mathbb{C})$ and they all are defined over $\Q(\zeta_N).$
\end{lemma}

\begin{proof}
    See discussion preceding Lemma 5.2 of \cite{2206.14959} and Lemma 5.2 of \cite{2206.14959} for a proof.
\end{proof}

\begin{lemma}\label{lemma:cuspaction}
    There is a Galois action of $~\Gal(\Q(\zeta_N)/\Q)$ on $SC$ defined as follows. For $\sigma \in \Gal(\Q(\zeta_N)/\Q)$ that sends $\zeta_N$ to $\zeta_N^a$ and $P=GAU_N \in X_G(\mathbb{C})$, define $\sigma \cdot P \coloneqq GA\begin{psmallmatrix}
    1 & 0 \\
    0 & a
\end{psmallmatrix}U_N.$
\end{lemma}

\begin{proof}
    See Lemma 5.2 of \cite{2206.14959} for a proof.
\end{proof}

To check if a cusp $P$ in $SC$ is a rational cusp, i.e., $P \in X_G(\Q)$ we need to check if $\sigma \cdot P=P$ for all $\sigma \in \Gal(\Q(\zeta_N)/\Q)$, where $\cdot$ is as defined in Lemma \ref{lemma:cuspaction}. The script \texttt{Cusps} that is available in this paper's github repository of \cite{RakviGitHub} computes the number of rational cusps in $X_G(\mathbb{C}).$




\section{Isogenies}\label{sec:isogenies}
In this section, we state some results about isogenies that will be used throughout this article.

\begin{definition}
    Let $E/\Q$ and $E'/\Q$ be elliptic curves. An isogeny $\phi \colon E \to E'$ is a surjective morphism that induces a group homomorphism $\phi \colon E(\Qbar) \to E'(\Qbar).$ In addition, if the kernel of $\phi$ is cyclic, we say that $\phi$ is cylic isogeny.
\end{definition}

\begin{definition}
    We say that two elliptic curves $E_1$ and $E_2$ are isogenous to each other if there exists an isogeny between them. The degree of the isogeny $\phi$ is defined as the order of its kernel.
\end{definition}

It can be shown that being isogenous is an equivalence relation. Fix a prime $p$ in the set $\{2,3,5,7,11,13,17,19,37,43,67,163\}.$ For rest of this section, let $E_1$ and $E_2$ be two elliptic curves defined over $\Q$ with a degree $p$ rational isogeny $\phi \colon E_1 \to E_2$ between them.

\begin{lemma}\label{lemma:isomorphismofisogenis}
    Let $\ell$ be a prime not equal to $p.$ Then, $\rho_{E_1,\ell^{\infty}}(\G_{\Q})$ is conjugate in $\GL_2(\ZZ_{\ell})$ to $\rho_{E_2,\ell^{\infty}}(\G_{\Q}).$ 
\end{lemma}

\begin{proof}
    For any positive integer $m$ greater than or equal to $1$, $\phi$ induces a homomorphism $\phi_m \colon E_1[\ell^m] \to E_2[\ell^m]$, so the kernel of $\phi_m$ is contained in the kernel of $\phi.$ Since the only prime that can divide the order of the kernel of $\phi_m$ is $\ell$, the kernel of $\phi_m$ must be trivial. We will now show that $E_2[\ell^m]$ and $E_1[\ell^m]$ have the same exact order, which will prove that $\phi_m$ is surjective. Let $\hat{\phi} \colon E_2 \to E_1$ be the unique dual isogeny of degree $p$. For any positive integer $m$ greater than or equal to $1$, $\hat{\phi}$ also induces a homomorphism $\hat{\phi}_m \colon E_2[\ell^m] \to E_1[\ell^m].$ Since the only prime that can divide the order of the kernel of $\hat{\phi}_m$ is $\ell$, kernel of $\hat{\phi}_m$ must also be trivial. It follows easily that $\phi_m$ is an isomorphism. Since $\phi$ is a rational isogeny, $E_1[\ell^m]$ and $E_2[\ell^m]$ are isomorphic as $\G_{\Q}$ modules. Therefore, $E_1[\ell^{\infty}]$ and $E_2[\ell^{\infty}]$ are also isomorphic as $\G_{\Q}$ modules. Let $\{P,Q\}$ be a basis for $E_1[\ell^{\infty}]$ chosen in a compatible way, i.e., $\rho_{E_1,\ell^{m}}=\pi_m \circ \rho_{E_1,\ell^{\infty}}.$ Since $\phi$ induces an isomorphism from $E_1[\ell^{\infty}]$ to $E_2[\ell^{\infty}]$, $\{\phi(P),\phi(Q)\}$ is a basis for $E_2[\ell^{\infty}].$ For any $\sigma \in \G_{\Q}$
    \begin{align*}
        \sigma(P) &= aP+bQ \\
        \sigma(Q) &= cP+dQ,
\end{align*} if and only if \begin{align*}
        \sigma(\phi(P))&=a\phi(P)+b\phi(Q) \\
        \sigma(\phi(Q))&=c\phi(P)+d\phi(Q).
    \end{align*} 
    
    Rewriting these equations as matrices, if $\rho_{E_1,\ell^{\infty}}(\sigma)$ is $\begin{psmallmatrix}
    a & c \\
    b & d
\end{psmallmatrix}$ with respect to basis $\{P,Q\}$, then $\rho_{E_2,\ell^{\infty}}(\sigma)$ is also $\begin{psmallmatrix}
    a & c \\
    b & d
\end{psmallmatrix}$ but with respect to basis $\{\phi(P),\phi(Q)\}$. The Lemma now follows.
\end{proof}

We set some notation before we state more results. From Theorem 1.1.6 of \cite{MR4468989} we know that there exists a smallest number $M_p$ such that $\rho_{E,p^{\infty}}(\G_{\Q})$ is determined modulo $M_p$ for all $E/\Q$ without CM. In particular, for $p=3$ $M_3=27.$ Further it follows from Theorem 1.1.6 of \cite{MR4468989} that each $M_p$ is a power of $p$. 

Fix a basis $\{P,Q\}$ for $E_1[M_p]$ satisfying $\ker(\phi) \subseteq \langle P \rangle.$

\begin{lemma}
    Assume the above setup.
    \begin{enumerate}
        \item Then, $\rho_{E_1,p}(\G_{\Q})$ is conjugate to a subgroup of upper triangular matrices of $\GL_2(\ZZ/p\ZZ).$

        \item The $p$-adic Galois image associated to $E_2$, i.e., $\rho_{E_2,p^{\infty}}(\G_{\Q})$ can be explicitly determined from $\rho_{E_1,p^{\infty}}(\G_{\Q}).$
    \end{enumerate}
     
\end{lemma}

\begin{proof}
\noindent
\begin{enumerate}
    \item Define $P' \coloneqq (M_p/p)P$ and $Q' \coloneqq (M_p/p)Q.$ 
    Since $\{P,Q\}$ is a basis for $E_1[M_p]$, we know that the order of $P'$ and $Q'$ is equal to $p$ and $\langle P' \rangle \intersect \langle Q' \rangle$ is trivial. 
    Hence, $\{P',Q'\}$ is a $\ZZ/M_p\ZZ$-basis for $E_1[p]$ because $\{P,Q\}$ is a basis for $E_1[M_p].$ Since, $\phi$ is a rational isogeny the kernel of $\phi$ is Galois stable. Since, $\ker(\phi) \subseteq \{P,2P,\ldots,M_pP\}$ for $\sigma \in \G_{\Q}$, we have $\sigma(P')=aP'$ for some $a \in (\ZZ/p\ZZ)^{\times}.$ The Lemma now follows.

    \item Choose $P''$ and $Q''$ such that they satisfy $pP''=P$ and $pQ''=Q.$ Since, $\{P,Q\}$ is a basis for $E_1[M_p]$ and $\ker(\phi) \subseteq \{P,2P,\ldots,M_pP\}$, so $\{\phi(P''),\phi(Q)\}$ is a basis for $E_2[M_p].$ This follows because $M_p \phi(P'')=\phi(M_p P'')=\phi((M_p/p) P)$ is identity since the order of the kernel is $p$, hence $M_p$ is also the order of $\phi(P'')$. Since the order of $Q$ is $M_p$, it follows that the order of $\phi(Q)$ is also $M_p$. Further, $\langle \phi(P'') \rangle \intersect \langle \phi(Q) \rangle$ is trivial because $\langle P \rangle \intersect \langle Q \rangle$ is trivial. For any $\sigma \in \G_{\Q}$ assume 
    \begin{align*}
        \sigma(P) &= aP+bQ \\
        \sigma(Q) &= cP+dQ.
    \end{align*} Multiplying first equation with $M_p/p$ we get, $\sigma(P') = aP'+bQ'$, from previous part of this Lemma we know that modulo $p$, $b$ is zero. If we multiply first equation with $1/p$, we get $\sigma(P'') = aP''+(b/p)Q$. If we substitute $P=pP''$ in $\sigma(Q) = cP+dQ$  we obtain $ \sigma(Q) = cpP''+dQ.$ Applying $\phi$ to

    \begin{align*}
        \sigma(P'') &= aP''+(b/p)Q \\
        \sigma(Q) &= cpP''+dQ
    \end{align*} 

    we get \begin{align}
        \label{eqn1:generators}
        \sigma(\phi(P''))&=a\phi(P'')+\frac{b}{p}\phi(Q) \\
        \label{eqn2:generators} \sigma(\phi(Q))&=pc\phi(P'')+d\phi(Q).
        \end{align}
    Rewriting these equations as matrices, if $\rho_{E_1,p^{\infty}}(\sigma)$ is $\begin{psmallmatrix}
    a & c \\
    b & d
\end{psmallmatrix}$ with respect to basis $\{P,Q\}$, then $\rho_{E_2,p^{\infty}}(\sigma)$ is $\begin{psmallmatrix}
    a & pc \\
    b/p & d
\end{psmallmatrix}$ with respect to basis $\{\phi(P''),\phi(Q)\}$.

So, we can explicitly construct generators of $\rho_{E_2,p^{\infty}}(\G_{\Q})$ from generators of $\rho_{E_1,p^{\infty}}(\G_{\Q})$ using \cref{eqn1:generators} and \cref{eqn2:generators}. 
\end{enumerate}

\end{proof}
We state the following Theorem on 3-adic images associated to elliptic curves without CM which follows from Theorem $1.1.6$ of \cite{MR4468989} and Theorem $1.3$ of \cite{2501.07833}. 
\begin{theorem}
\noindent
\begin{enumerate}
    \item \label{thm:no3isogeny} Let $E/\Q$ be an elliptic curve without CM such that $E$ does not admit a degree 3 isogeny. Then, either $\rho_{E,3^{\infty}}(\G_{\Q})$ is $\GL_2(\ZZ_3)$ or the RSZB label of $\rho_{E,3^{\infty}}(\G_{\Q})$ lies in the set \{3.3.0.1, 3.6.0.1, 9.9.0.1, 9.18.0.1, 9.18.0.2, 9.27.0.1, 9.27.0.2\}.

    \item \label{thm:3isogeny} Let $E/\Q$ be an elliptic curve without CM such that $E$ admits a degree 3 isogeny. Then, the RSZB label of $\rho_{E,3^{\infty}}(\G_{\Q})$ lies in the set \{3.4.0.1, 3.8.0.1, 3.8.0.2, 3.12.0.1, 3.24.0.1, 9.12.0.1, \\
    9.12.0.2, 9.24.0.1-9.24.0.4, 9.36.0.1 - 9.36.0.9, 9.72.0.1-9.72.0.16, 27.36.0.1, 27.72.0.1, 27.72.0.2\}.

\end{enumerate}
\end{theorem}

For the rest of this article we will use RSZB labels to denote corresponding subgroups.

In the table below, we explicitly list $\rho_{E_2,3^{\infty}}(\G_{\Q})$ given $\rho_{E_1,3^{\infty}}(\G_{\Q}).$ Here we assume that there is a degree $3$ isogeny between $E_1$ and $E_2.$ The code for verifying entries of Table \ref{table:3adicimages} can be found in script \texttt{Code for Table 1} of \cite{RakviGitHub}.
\newpage
\begin{table}[H] 
\caption{Table for reading $\rho_{E_2,3^{\infty}}(\G_{\Q})$ from  $\rho_{E_1,3^{\infty}}(\G_{\Q})$, when there is a degree $3$ isogeny between $E_1$ and $E_2$}\label{table:3adicimages}
\renewcommand{\arraystretch}{1.3}
\vspace{0.5 cm}
\begin{tabular}{|l|l|l|l|}
\hline
$\rho_{E_1,3^{\infty}}(\G_{\Q})$    & $\rho_{E_2,3^{\infty}}(\G_{\Q})$ & $\rho_{E_1,3^{\infty}}(\G_{\Q})$    & $\rho_{E_2,3^{\infty}}(\G_{\Q})$ \\ \hline
3.4.0.1 & 3.4.0.1   & 3.12.0.1 & 9.12.0.1       \\ \hline
9.12.0.1 & 3.12.0.1  & 9.12.0.2 & 9.12.0.2\\ \hline
9.36.0.1 & 9.36.0.4 & 9.36.0.2 & 27.36.0.1  \\ \hline
9.36.0.3 & 9.36.0.6 & 9.36.0.4 & 9.36.0.1 \\ \hline
9.36.0.5 & 9.36.0.1 & 9.36.0.6 & 9.36.0.3 \\ \hline
9.36.0.7 & 9.36.0.9 & 9.36.0.8 & 9.36.0.8\\\hline
9.36.0.9 & 9.36.0.7 & 27.36.0.1 & 9.36.0.2\\ \hline
3.8.0.1 & 3.8.0.2 & 3.8.0.2 & 3.8.0.1\\\hline
3.24.0.1 & 9.24.0.1 & 9.24.0.1 & 3.24.0.1 \\ \hline
9.24.0.2 & 9.24.0.4 & 9.24.0.3 & 3.24.0.1 \\ \hline
9.24.0.4 & 9.24.0.2 & 9.72.0.1 & 9.72.0.5 \\ \hline
9.72.0.2 & 9.72.0.6 & 9.72.0.3 & 27.72.0.1 \\ \hline
9.72.0.4 & 9.72.0.7 & 9.72.0.5 & 9.72.0.1 \\ \hline
9.72.0.6 & 9.72.0.2 & 9.72.0.7 & 9.72.0.4 \\ \hline
9.72.0.8 & 9.72.0.16 & 9.72.0.9 & 9.72.0.15 \\ \hline
9.72.0.10 & 9.72.0.14 & 9.72.0.11 & 9.72.0.2 \\ \hline
9.72.0.12 & 9.72.0.1 & 9.72.0.13 & 9.72.0.4 \\ \hline
9.72.0.14 & 9.72.0.10 & 9.72.0.15 & 9.72.0.9 \\ \hline
9.72.0.16 & 9.72.0.8 & 27.72.0.1 & 9.72.0.3 \\ \hline
27.72.0.2 & 9.72.0.3 & & \\ \hline

\end{tabular}
\end{table}

\section{Analysis of rational points on various modular curves}
\label{sec:Analysisofrationalpoints}

We give a summary of various techniques used to compute the set of rational points denoted by $C(\Q)$ where $C$ is a modular curve defined over $\Q$ of our interest.

\begin{itemize}

    \item If $C$ is an elliptic curve of rank $0$, then we can compute $C(\Q)$ by using Nagell-Lutz Theorem.
    
    \item When genus of $C$ is $2$ and rank of Jacobian of $C$ denoted by $J_C$ is $0$, then $J_C(\Q)=J_C(\Q)_{\tors}$ and we can get $C(\Q)$ by computing preimages of points in $J_C(\Q)_{\tors}$ under the Abel-Jacobi map $C \to J_C$ which is implemented in \texttt{MAGMA} \cite{MR1484478} as \texttt{Chabauty0} command.

    \item For some curves, we compute the canonical model and observe that there are no $\Q_p$ or $\F_p$ solutions.
    
    \item In some cases there is a morphism $\pi \colon C \to C'$ and we can completely determine $C'(\Q)$, so we can find $C(\Q)$ by lifting via $\pi.$

    \item Some curves admit an involution $\iota$ such that the map $P \to P-\iota(P)$ projects onto the torsion subgroup of Jacobian and we were able to "sieve" with respect to this projection. This is based on method described in Section 8.3 of \cite{MR4468989}.
\end{itemize}

For rank of Jacobian, we used the information available at \url{https://beta.lmfdb.org/} \cite{LMFDB} which is based on appendix in \cite{MR4468989}. After determining $C(\Q)$ we use script \texttt{Cusps} to compute the number of rational cusps on $C.$ If there are still any rational points left, then we refer to \url{https://beta.lmfdb.org/} to rule out CM points. If there are still any rational points left, then we looked at non CM elliptic curves across \cite{LMFDB}, computed their adelic image using script \texttt{FindOpenImage.m} of \cite{2206.14959} to analyze which ones give rise to non CM non cuspidal rational points on $C.$

\subsection{Elliptic curves that do not admit degree 3 isogeny}

From the classification of isogeny-torsion graphs of elliptic curves defined over $\Q$ \cite{chiloyan2021classification}, we know that if $E/\Q$ does not admit a degree 3 isogeny, then $E(\Q)_{\tors}$ is either trivial or is isomorphic to one of the groups in set \[\{\ZZ/2\ZZ, \ZZ/4\ZZ, \ZZ/5\ZZ, \ZZ/7\ZZ, \ZZ/8\ZZ, \ZZ/10\ZZ, \ZZ/2\ZZ \times \ZZ/2\ZZ, \ZZ/2\ZZ \times \ZZ/4\ZZ, \ZZ/2\ZZ \times \ZZ/8\ZZ\}.\] In this section, we rule out these possibilities when we have constraints on 3-adic image.

\begin{lemma}\label{lemma:3.3.0.1}
\noindent
Let $E/\Q$ be without CM. If $\rho_{E,3^{\infty}}(\G_{\Q})$ is conjugate to a subgroup of 3.3.0.1, then $E(\Q)_{\tors}$ cannot contain an order 4 subgroup, cannot contain a point of order 5 and cannot admit a degree 7 isogeny, therefore cannot have a rational point of order 7.  
\end{lemma}

\begin{proof}
 We consider fiber products of $X_{3.3.0.1}$ with $X(2)$ and $X_0(4)$ respectively. Both have genus $1$, and their Weierstrass model is given by $E \colon y^2=x^3+1$ which has rank $0.$ Using Nagell-Lutz Theorem we compute $E(\Q)$ which is equal to $\{(0,1,0),(0,1,1),(0,-1,1),(-1,0,1),(2,-3,1),(2,3,1)\}$ out of which there are three rational cusps in both cases. In first case, there are three CM points $\{(-1,0,1),(2,-3,1),(2,3,1)\}$ that correspond to \href{https://lmfdb.org/EllipticCurve/Q/32/a/3}{32.a3}. In second case, there are three CM points $\{(-1,0,1),(2,-3,1),(2,3,1)\}$ out of which $(-1,0,1)$ corresponds to \href{https://lmfdb.org/EllipticCurve/Q/32/a/3}{32.a3} and $\{(2,-3,1),(2,3,1)\}$ correspond to \href{https://lmfdb.org/EllipticCurve/Q/32/a/1}{32.a1}.
 
 The fiber product of $X_{3.3.0.1}$ and $X_{\pm1}(5)$ has genus $2$, so it is hyperelliptic with Weierstrass model given by $C \colon y^2+(x^3+1)y=5x^3+31$ and its Jacobian has rank $0.$ Using the command \texttt{Chabauty0} we know that there are only two rational points on $C$ and both are cusps.

The fiber product of $X_{3.3.0.1}$ and $X_{0}(7)$ has genus $2$, so it is hyperelliptic with Weierstrass model given by $C \colon y^2+(x^3+1)y=6x^3-7$ and its Jacobian has rank $0.$ Using the command \texttt{Chabauty0} we know that there are only four rational points on $C$ out of which there are two rational cusps and two rational CM points $\{(1,-1,1),(-3,13,1)\}$ that correspond to \href{https://lmfdb.org/EllipticCurve/Q/49/a/2}{49.a2} and \href{https://lmfdb.org/EllipticCurve/Q/49/a/1}{49.a1} respectively.

   \end{proof}

\begin{Corollary}\label{cor:3.3.0.1} Let $E/\Q$ be without CM.
\noindent
\begin{enumerate}
    \item If $\rho_{E,3^{\infty}}(\G_{\Q})$ is conjugate to a subgroup of 3.6.0.1, then $E(\Q)_{\tors}$ cannot contain a subgroup of order 4 or a point of order 5 and cannot admit an isogeny of degree 7.
    
    \item If $\rho_{E,3^{\infty}}(\G_{\Q})$ is conjugate to a subgroup of 9.9.0.1, then $E(\Q)_{\tors}$ cannot contain a point of order 5 and cannot admit an isogeny of degree 7.

    \item If $\rho_{E,3^{\infty}}(\G_{\Q})$ is conjugate to a subgroup of 9.18.0.1, then $E(\Q)_{\tors}$ cannot contain a point of order 5 and cannot admit an isogeny of degree 7.
\end{enumerate}

\end{Corollary}

\begin{proof}

    Each group in the set \{3.6.0.1, 9.9.0.1, 9.18.0.1\} is a subgroup of 3.3.0.1. Proof now follows from Lemma \ref{lemma:3.3.0.1}.
\end{proof}

\begin{lemma}\label{lemma:9.9.0.1}
   Let $E/\Q$ be without CM. If $\rho_{E,3^{\infty}}(\G_{\Q})$ is conjugate to a subgroup of 9.9.0.1, then $E(\Q)_{\tors}$ cannot contain a point of order 2. 
\end{lemma}

\begin{proof}

The fiber product of $X_{9.9.0.1}$ and $X_0(2)$ has genus $1$ with Weierstrass model given by $E \colon y^2+(x+1)y=x^3-x^2+4x-1$ and has rank $0.$ Using Nagell-Lutz Theorem we explicitly compute $E(\Q)=\{(0,1,0),(1,1,1),(1,-3,1)\}$ out of which two are cusps and $(1,-3,1)$ is a rational CM point that corresponds to \href{https://lmfdb.org/EllipticCurve/Q/256/a/1}{256.a1}. So, the proof follows.
\end{proof}

\begin{Corollary}\label{cor:9.18.0.2}
    Let $E/\Q$ be without CM. If $\rho_{E,3^{\infty}}(\G_{\Q})$ is conjugate to a subgroup of 9.18.0.2, then $E(\Q)_{\tors}$ cannot contain a point of order 2, a point of order 5 and cannot admit an isogeny of degree 7.
    \end{Corollary}

\begin{proof}
    The group 9.18.0.2 is a subgroup of 9.9.0.1. Proof now follows from Lemma \ref{lemma:9.9.0.1} and Corollary \ref{cor:3.3.0.1}.
\end{proof}

\begin{lemma}\label{lemma:9.18.0.1}
  Let $E/\Q$ be without CM. If $\rho_{E,3^{\infty}}(\G_{\Q})$ is conjugate to a subgroup of 9.18.0.1, then $E(\Q)_{\tors}$ cannot contain a point of order 2.
\end{lemma}

\begin{proof}
The fiber product of $X_{9.18.0.1}$ and $X_0(2)$ has genus $1$ with Weierstrass model given by $E \colon y^2+(x+1)y=x^3-x^2-5x+5$ and has rank $0.$ Using Nagell-Lutz Theorem we compute that $E(\Q)=\{(0,1,0), (1,0,1),(1,-2,1)\}$ out of which two are cusps and $(0,1,0)$ is a rational CM point that corresponds to \href{https://lmfdb.org/EllipticCurve/Q/256/a/1}{256.a1}.
\end{proof}

\begin{lemma}\label{lemma:9.27.0.1}
    Let $E/\Q$ be without CM. If $\rho_{E,3^{\infty}}(\G_{\Q})$ is conjugate to a subgroup of 9.27.0.1, then $E(\Q)_{\tors}$ cannot contain a point of order 2, cannot admit an isogeny of degree 5 or 7. 
\end{lemma}

\begin{proof}
The fiber product of $X_{9.27.0.1}$ with $X_0(2)$ has no $\Q_p$ point for $p=3$, so it cannot have rational points. We verify this by computing all points modulo $3$ and then checking that none of those lift. The fiber product of $X_{9.27.0.1}$ with $X_0(5)$ has no points modulo $p=7$ and the fiber product of $X_{9.27.0.1}$ with $X_0(7)$ has no points modulo $p=2$. These computations can be verified using script \texttt{9.27.0.1} of \cite{RakviGitHub}.
\end{proof}

\begin{lemma}\label{lemma:9.27.0.2}
    Let $E/\Q$ be without CM. If $\rho_{E,3^{\infty}}(\G_{\Q})$ is conjugate to a subgroup of 9.27.0.2, then $E(\Q)_{\tors}$ cannot contain a point of order 2, cannot admit an isogeny of degree 5 or 7. 
\end{lemma}

\begin{proof}
  The fiber product of $X_{9.27.0.2}$ with $X_0(5)$ has no $\Q_3$ points. We verify this by computing all points modulo $3$ and then checking that none of those lift. This can be verified using script \texttt{9.27.0.2 x X0(5)} of \cite{RakviGitHub}. 
  
  Since 9.27.0.2 is a subgroup of 3.3.0.1, by Lemma \ref{lemma:3.3.0.1} it follows that there are no non cuspidal non CM rational points on fiber product of $X_H$ and $X_0(7).$ 
  
  We will now analyze the rational points on the fiber product $C$ of $X_{9.27.0.2}$ and $X_0(2)$ which has LMFDB label \href{https://beta.lmfdb.org/ModularCurve/Q/18.81.3.a.1/}{18.81.3.1}. Its Jacobian has rank $1.$ The canonical model for $C$ is given by equation\[         2x^4+2x^3y-xy^3-2x^3z-3x^2yz+2y^3z-3y^2z^2+xz^3+2yz^3=0.\] It is easy to observe an automorphism $i \colon C \to C$ given by 
         \begin{align*}
             [x,y,z] \to [x-z,-y].
         \end{align*} 

The only rational point that is fixed under $i$ is $(1,1,-1).$

The quotient of $C$ by $i$ is an elliptic curve $E$ of rank $1.$ Therefore, there exists an abelian variety of rank $0$, say $V$ such that the Jacobian of $C$, denoted by $J_C$ decomposes into $E \times V$ since rank of $J_C$ is also $1.$ For any point $P \in C(\Q)$, the point $P-i(P)$ lies in $J_C(\Q)_{\tors}.$ Hence, it suffices to compute preimages of rational points in $J_{C_{\tors}}$ under the map $a \colon C(\Q) \to J_C(\Q)_{\tors}$ given by $P \mapsto P-i(P)$ which is injective away from the fixed points of $i.$ We were not able to compute $J_C(\Q)_{\tors}$ but using local computations which we describe below we were able to deduce that $J_C(\Q)_{\tors}$ is a subgroup of $\ZZ/3\ZZ \times \ZZ/9\ZZ$. Let $p$ be an odd prime not equal to $2$ or $3$. From the Lemma in appendix of \cite{MR604840} we know that $J_C(\Q)_{\tors}$ injects into $J_C(\ZZ/p\ZZ).$ Using \texttt{ClassGroup} command of \texttt{MAGMA} for primes $5$, $13$ and $73$ we find that $J_C(\Q)_{\tors}$ must be a subgroup of $\ZZ/3\ZZ \times \ZZ/9\ZZ.$ We checked that for any two points $P,Q$ in reduction of $C$ modulo $7$, $P-i(P)$ is not of order 3. There are two points in $C(\Q)$ that are easy to observe, $P_1=(0,0,1)$ and $P_2=(0,1,0)=i(P_1)$. Define $D \coloneqq P_1-P_2$, it is a divisor of order 9. So now we need to show that for each $a \in \{1,2,3,4,5,6,7,8\}$ there is either a rational point $P \in C(\Q)$ such that $P-i(P)$ is equivalent to $aD$ or there exists a prime $p$ for which there is no point $P$ in $C(\ZZ/p\ZZ)$ such that $P-i(P)$ is equivalent to $aD.$ Working modulo $7$, we verify that the divisor $(1,1,1)-(1,-1,-1)$ is equivalent to $4D$, the divisor $(-1,1,1)-(1,1,1)$ is equivalent to $5D$ and the divisor $(0,1,0)-(0,0,1)$ is equivalent to $8D.$ For all the other values of $a$, there is no point $P$ in $C(\ZZ/7\ZZ)$ such that $P-i(P)$ is equivalent to $aD.$ So, $C(\Q)=\{(1,1,-1),(0,1,0),(0,0,1),(1,1,1),(1,-1,-1)\}.$ This can be verified using script \texttt{9.27.0.2 x X0(2)} of \cite{RakviGitHub}.

There are no rational cusps. All the points are rational CM points, $\{(1,-1,-1),(1,1,-1)\}$ correspond to \href{https://lmfdb.org/EllipticCurve/Q/32/a/3}{32.a3}, $(0,0,1)$ corresponds to \href{https://lmfdb.org/EllipticCurve/Q/49/a/2}{49.a2}, $(1,1,1)$ corresponds to \href{https:/lmfdb.org/EllipticCurve/Q/32/a/1}{32.a1} and $(0,1,0)$ corresponds to \href{https://lmfdb.org/EllipticCurve/Q/49/a/1}{49.a1}.
\end{proof}

\subsection{Elliptic curves that admit a degree 3 isogeny}

From the classification of isogeny-torsion graphs of elliptic curves defined over $\Q$ \cite{chiloyan2021classification}, we know that if $E/\Q$ admits a degree 3 isogeny then $E(\Q)_{\tors}$ is either trivial or is isomorphic to one of the groups in \[\{\ZZ/2\ZZ, \ZZ/3\ZZ, \ZZ/4\ZZ, \ZZ/5\ZZ, \ZZ/6\ZZ, \ZZ/9\ZZ, \ZZ/12\ZZ, \ZZ/2\ZZ \times \ZZ/6\ZZ, \ZZ/2\ZZ \times \ZZ/2\ZZ\}.\] In this section, we rule out these possibilities when we have constraints on 3-adic image.

\begin{proposition}\label{prop:X0(15)}
    Let $E/\Q$ be without CM. If $E$ admits a degree 3 isogeny and a degree 5 isogeny, then $\rho_{E,3^{\infty}}(\G_{\Q})$ lies in the set \{3.4.0.1, 3.8.0.1, 3.8.0.2\}.
\end{proposition}

\begin{proof}
    If $E$ admits a degree 3 isogeny and a degree 5 isogeny, then $\rho_E(\G_{\Q})$ modulo $15$ is contained in Borel subgroup of $\GL_2(\ZZ/15\ZZ).$ From Lemma \ref{lemma:modular curves} we know that $j(E) \in \pi(X_0(15)).$ The curve $X_0(15)$ has genus $1$, rank $0$. Its Weierstrass model is $y^2+(x+1)y=x^3+x^2-10x-10$ and it has eight rational points. There are four rational cusps and four non cuspidal non CM rational points that correspond to elliptic curves \href{https://lmfdb.org/EllipticCurve/Q/50/a/1}{50.a1}, \href{https://lmfdb.org/EllipticCurve/Q/50/a/2}{50.a2}, \href{https://lmfdb.org/EllipticCurve/Q/50/a/3}{50.a3} and \href{https://lmfdb.org/EllipticCurve/Q/50/a/4}{50.a4}. Therefore, $E$ is a quadratic twist of one of the above elliptic curves and hence, $\rho_{E,3^{\infty}}(\G_{\Q})$ lies in the set \{3.4.0.1, 3.8.0.1, 3.8.0.2\}.
\end{proof}

\begin{lemma}\label{no5torsion3.8.0.2}
Let $E/\Q$ be without CM. If $\rho_{E,3^{\infty}}(\G_{\Q})$ is conjugate to 3.8.0.2, then $E(\Q)_{\tors}$ cannot be isomorphic to $\ZZ/5\ZZ.$
\end{lemma}

\begin{proof}
    
    Let $E$ be such an elliptic curve. Define $H \coloneqq 3.8.0.2$ and $H_0 \colon=\langle H,-I \rangle$. Then, $E$ corresponds to a non cuspidal rational point on $X_{H_0}$. From Lemma $5.1$ of \cite{MR3500996} we know that there exists a quadratic twist $E_d$ such that $\rho_{E_d,3^{\infty}}(\G_{\Q})$ is conjugate to $H'$, where $H'$ is 3.8.0.1. Any elliptic curve $E$ over $\Q$ whose 3-adic image is conjugate to $H'$ must have a rational point of order 3, which is a contradiction because from Proposition 1(b) of \cite{MR3383431} we know that $E_d$ must have trivial torsion.
\end{proof}

\begin{lemma}\label{lemma:3.12.0.1}
    Let $E/\Q$ be without CM. If $\rho_{E,3^{\infty}}(\G_{\Q})$ is conjugate to a subgroup of 3.12.0.1, then $E(\Q)_{\tors}$ cannot contain a subgroup of order 4.  
\end{lemma}

\begin{proof}
    The group 3.12.0.1 is a subgroup of 3.3.0.1. Proof now follows from Lemma \ref{lemma:3.3.0.1}.
\end{proof}

\begin{lemma}\label{lemma:9.12.0.1order4}
Let $E/\Q$ be without CM. If $\rho_{E,3^{\infty}}(\G_{\Q})$ is conjugate to a subgroup of 9.12.0.1, then $E(\Q)_{\tors}$ cannot contain a subgroup of order 4.  
\end{lemma}

\begin{proof}
Let $E/\Q$ be without CM such that $\rho_{E,3^{\infty}}(\G_{\Q})$ is conjugate to 9.12.0.1 and $E(\Q)_{\tors}$ is of order 4. If $E$ has a point of order $4$, then $E$ gives rise to a non cuspidal rational point on $X_0(36)$ which is not possible because $X_0(36)$ is an elliptic curve with Weierstrass equation $y^2=x^3+1$ that has six rational points all of which are cusps. Thus, $E(\Q)_{\tors}$ is isomorphic to $\ZZ/2\ZZ \times \ZZ/2\ZZ$. However, this is not possible because from classification of isogeny-torsion graphs for elliptic curves over $\Q$ \cite{chiloyan2021classification} we know that if $E(\Q)_{\tors}$ is isomorphic to $\ZZ/2\ZZ \times \ZZ/2\ZZ$ then it cannot admit a degree 9 isogeny.

\end{proof}

\begin{lemma}\label{lemma:9.12.0.2tors2}
Let $E/\Q$ be without CM. If $\rho_{E,3^{\infty}}(\G_{\Q})$ is conjugate to a subgroup of 9.12.0.2, then $E(\Q)_{\tors}$ cannot be of order 2.  
\end{lemma}

\begin{proof}
    Suppose there exists an $E/\Q$ without CM such that $\rho_{E,3^{\infty}}(\G_{\Q})$ is conjugate to a subgroup of 9.12.0.2 and $E(\Q)_{\tors}$ is of order 2. Since $E$ has a torsion point of order $2$, it admits an isogeny of order $2$. So, $E$ must correspond to a non CM rational point on the fiber product of $X_{9.12.0.2}$ and $X_0(2).$ The fiber product $C$ of $X_{9.12.0.2}$ and $X_0(2)$ has genus $2$, hence it is hyperelliptic. It has rank $0.$ Using the command \texttt{Chabauty0}, we can verify that there are six rational points on $C$ out of which there are four rational cusps and two CM points that come from curves \href{https://lmfdb.org/EllipticCurve/Q/27/a/3}{27.a3} and \href{https://lmfdb.org/EllipticCurve/Q/36/a/1}{36.a1}. Hence, such an $E$ cannot exist.
\end{proof}

\begin{Corollary}\label{lemma:9.36.0.7-9.360.9order2}
    
Let $E/\Q$ be without CM. If $\rho_{E,3^{\infty}}(\G_{\Q})$ is conjugate to a subgroup of any group in the set \{9.36.0.7, 9.36.0.8, 9.36.0.9\}, then $E$ cannot have a rational point of order 2.

\end{Corollary}

\begin{proof}
    Each group in the set \{9.36.0.7, 9.36.0.8, 9.36.0.9 \} is a subgroup of 9.12.0.2. Proof now follows from Lemma \ref{lemma:9.12.0.2tors2}.

\end{proof}

\begin{lemma}\label{lemma:9.36.0.1-9.36.0.627.36.0.1order2}
     Let $E/\Q$ be without CM. If $\rho_{E,3^{\infty}}(\G_{\Q})$ is conjugate to a subgroup of any group in \{27.36.0.1, 9.36.0.1, 9.36.0.2, 9.36.0.3, 9.36.0.4, 9.36.0.5, 9.36.0.6\}, then $E$ cannot have a rational point of order 2.

\end{lemma}

\begin{proof}
Let $H$ be in \{27.36.0.1, 9.36.0.1, 9.36.0.2, 9.36.0.3, 9.36.0.4, 9.36.0.5, 9.36.0.6\}.
If $E/\Q$ is without CM such that $\rho_{E,3^{\infty}}(\G_{\Q})$ is conjugate to a subgroup of $H$ and has a rational point of  order 2, then $E$ gives rise to a non cuspidal rational point on the fiber product of $X_H$ and $X_0(2).$ All of these curves have genus $2$, their Jacobian has rank $0.$ Using command \texttt{Chabauty0} we get all the rational points and verify that either there are no rational points or all the rational points are cusps. These computations can be verified using script \texttt{9.36.0.1 to 9.36.0.6 27.36.0.1 Order 2} of \cite{RakviGitHub}.
\end{proof}

\begin{lemma}\label{lemma:torsion3}
    Let $E/\Q$ be such that it admits a degree 3 isogeny. Then, $E(\Q)_{\tors}$ contains a point of order 3 if and only if $\rho_{E,3^{\infty}}(\G_{\Q})$ is conjugate to one of the groups in \{3.8.0.1, 3.24.0.1, 9.24.0.1, 9.24.0.2, 9.72.0.1, 9.72.0.2, 9.72.0.3, 9.72.0.4, 9.72.0.5, 9.72.0.6, 9.72.0.7, 9.72.0.8, 9.72.0.9, 9.72.0.10, 27.72.0.1\}.
\end{lemma}

\begin{proof}

Using Theorem \ref{thm:3isogeny} we know the possibilities for $\rho_{E,3^{\infty}}(\G_{\Q})$ if $E$ admits a degree 3 isogeny. Then, $E(\Q)_{\tors}$ contains a point of order 3 if and only if $\rho_{E,3}(\G_{\Q})$ is conjugate to the subgroup of $\GL_2(\ZZ/3\ZZ)$ generated by matrices of the form $\big(\begin{smallmatrix}
1 & a \\
0 & b
\end{smallmatrix}\big).$ The groups that fulfill this condition are in the set \{3.8.0.1, 3.24.0.1, 9.24.0.1, 9.24.0.2, 9.72.0.1, 9.72.0.2, 9.72.0.3, 9.72.0.4, 9.72.0.5, 9.72.0.6, 9.72.0.7, 9.72.0.8, 9.72.0.9, 9.72.0.10, 27.72.0.1\}.

\end{proof}

\begin{lemma}\label{lemma:torsion9}
    Let $E/\Q$ be such that it admits a degree 3 isogeny. Then, $E(\Q)_{\tors}$ is isomorphic to $\ZZ/9\ZZ$ if and only if $\rho_{E,3^{\infty}}(\G_{\Q})$ is conjugate to 9.72.0.5.
\end{lemma}

\begin{proof}
We know that $E(\Q)_{\tors}$ contains a point of order 9 if and only if $\rho_{E,9}$ is conjugate to the subgroup of $\GL_2(\ZZ/9\ZZ)$ generated by matrices of the form $\big(\begin{smallmatrix}
1 & a \\
0 & b
\end{smallmatrix}\big).$ We check that the only group among groups in Theorem \ref{thm:3isogeny} that satisfies this condition is 9.72.0.5.
\end{proof}

\section{Classification of 3-adic Galois images of isogeny-torsion graphs}\label{sec:results}

The aim of this section is to give proof of Theorem \ref{thm:mainthm}.

\subsection{Linear graphs} 

We say that an isogeny-torsion graph is of type $L_1$, if there is exactly one vertex in it.

\begin{proposition}\label{prop:L_1}
    Let $E/\Q$ be without CM. If the isogeny-torsion graph associated to $E$ is of type $L_1$, then $\rho_{E,3^{\infty}}(\G_{\Q})$ is in the set \{~$\GL_2(\ZZ_3)$, 3.3.0.1, 3.6.0.1, 9.9.0.1, 9.18.0.1, 9.18.0.2, 9.27.0.1, 9.27.0.2\}.
\end{proposition}

\begin{proof}
     If the isogeny-torsion graph associated to $E$ is of type $L_1$, then in particular $E$ does not admit a degree 3 isogeny. So from Theorem \ref{thm:no3isogeny} we know the possibilities for $\rho_{E,3^{\infty}}(\G_{\Q}).$  Each of those possibilities occur, see \href{https://lmfdb.org/EllipticCurve/Q/37/a/}{37.a}, \href{https://lmfdb.org/EllipticCurve/Q/245/a/}{245.a}, \href{https://lmfdb.org/EllipticCurve/Q/1210/d/}{1210.d}, \href{https://lmfdb.org/EllipticCurve/Q/864/e/}{864.e}, \href{https://lmfdb.org/EllipticCurve/Q/2057/c/}{2057.c}, \href{https://lmfdb.org/EllipticCurve/Q/22898/d/}{22898.d}, \href{https://lmfdb.org/EllipticCurve/Q/1944/e/}{1944.e} and \href{https://lmfdb.org/EllipticCurve/Q/118579/b/}{118579.b}.
\end{proof}

We say that an isogeny graph is of type $L_2(p)$ if there are eaxctly two vertices in it connected with an edge of label $p$, where $p \in \{2,3,5,7,11,13,17,37\}.$

\begin{proposition}
    Let $E/\Q$ be without CM. If the isogeny-torsion graph associated to $E$ is of type $L_2(2)$, then either $\rho_{E,3^{\infty}}(\G_{\Q})$ is $\GL_2(\ZZ_3)$ or $\rho_{E,3^{\infty}}(\G_{\Q})$ lies in the set \{3.3.0.1, 3.6.0.1\}.
\end{proposition}

\begin{proof}
    If the isogeny-torsion graph associated to $E$ is of type $L_2(2)$, then each of the vertices have conjugate 3-adic Galois image from \cref{lemma:isomorphismofisogenis}. So, it suffices to classify possibilities for a single vertex. From \cref{lemma:9.9.0.1}, \cref{lemma:9.18.0.1}, \cref{cor:9.18.0.2}, \cref{lemma:9.27.0.1} and \cref{lemma:9.27.0.2} we know that either $\rho_{E,3^{\infty}}(\G_{\Q})$ is $\GL_2(\ZZ_3)$ or $\rho_{E,3^{\infty}}(\G_{\Q})$ is in the set \[\{3.3.0.1, 3.6.0.1\}.\] All of these possibilities occur, see \href{https://lmfdb.org/EllipticCurve/Q/46/a/}{46.a}, \href{https://lmfdb.org/EllipticCurve/1568/a/}{1568.a} and \href{https://lmfdb.org/EllipticCurve/Q/726/b/}{726.b}.
\end{proof}

\begin{proposition}
    Let $E/\Q$ be without CM. If the isogeny-torsion graph associated to $E$ is of type $L_2(3)$, then one of the following holds.
\begin{enumerate}
    \item If both vertices have trivial torsion then $(\rho_{E_i,3^{\infty}}(\G_{\Q}))_{i=1}^{i=2}$ is given by one of the tuples (3.4.0.1,3.4.0.1), (9.12.0.2,9.12.0.2), (9.36.0.7,9.36.0.9), (9.36.0.8,9.36.0.8).

    \item If only one vertex has trivial torsion, then $(\rho_{E_i,3^{\infty}}(\G_{\Q}))_{i=1}^{i=2}$ is given by one of the tuples (3.8.0.1,3.8.0.2), (9.24.0.2,9.24.0.4), (9.72.0.8,9.72.0.16), (9.72.0.9,9.72.0.15), (9.72.0.10,9.72.0.14).
\end{enumerate}
\end{proposition}

\begin{proof} We know that both $E_1$ and $E_2$ admit a degree 3 isogeny and do not admit a degree 9 cyclic isogeny.
    \begin{enumerate}
        \item If both the vertices have trivial torsion, then combining Theorem \ref{thm:3isogeny}, Lemma \ref{lemma:torsion3} we get that $\rho_{E_i,3^{\infty}}(\G_{\Q})$ for $i \in \{1,2\}$ lies in \{3.4.0.1, 3.8.0.2, 3.12.0.1, 9.12.0.2, 9.24.0.4, 9.36.0.1, 9.36.0.2, 9.36.0.3, 9.36.0.7, 9.36.0.8, 9.36.0.9, 9.72.0.14, 9.72.0.15, 9.72.0.16\}. From Table \ref{table:3adicimages} the only pair of labels that lie in this set are (3.4.0.1,3.4.0.1), (9.12.0.2,9.12.0.2), (9.36.0.7,9.36.0.9), (9.36.0.8,9.36.0.8). See, \href{https://lmfdb.org/EllipticCurve/176/a/}{176.a}, \href{https://lmfdb.org/EllipticCurve/196/a/}{196.a}, \href{https://lmfdb.org/EllipticCurve/1734/k/}{1734.k} and \href{https://lmfdb.org/EllipticCurve/17100/r/}{17100.r} respectively.

        \item Assume that $E_1(\Q)_{\tors}$ is of order 3, then from Lemma \ref{lemma:torsion3} and the fact that $E_1$ does not admit cyclic isogeny of degree 9, $\rho_{E_1,3^{\infty}}(\G_{\Q})$ lies in \{3.8.0.1, 3.24.0.1, 9.24.0.2, 9.72.0.1, 9.72.0.2, 9.72.0.3, 9.72.0.4, 9.72.0.8, 9.72.0.9, 9.72.0.10\}. Since $E_2$ has trivial torsion and it does not admit cyclic isogeny of degree 9, we are left with following possibilities of $\rho_{E_1,3^{\infty}}(\G_{\Q})$, \{3.8.0.1, 9.24.0.2, 9.72.0.8, 9.72.0.9, 9.72.0.10\}. From Table \ref{table:3adicimages} $(\rho_{E_i,3^{\infty}}(\G_{\Q}))_{i=1}^{i=2}$ is given by one of the tuples (3.8.0.1,3.8.0.2), (9.24.0.2,9.24.0.4), (9.72.0.8,9.72.0.16), (9.72.0.9,9.72.0.15), (9.72.0.10,9.72.0.14). See \href{https://lmfdb.org/EllipticCurve/44/a/}{44.a}, \href{https://lmfdb.org/EllipticCurve/196/b/}{196.b}, \href{https://lmfdb.org/EllipticCurve/486/d/}{486.d}, \href{https://lmfdb.org/EllipticCurve/17100/j/}{17100.j} and \href{https://lmfdb.org/EllipticCurve/486/c/}{486.c} respectively.
    \end{enumerate}
\end{proof}

\begin{lemma}\label{lemma:no5isogeny}
    Let $E/\Q$ be a without CM. If $\rho_{E,3^{\infty}}(\G_{\Q})$ is in \{9.18.0.1,9.18.0.2\}, then $E$ cannot admit a degree 5 isogeny.
\end{lemma}

\begin{proof}
    We look at the fiber product $C$ of $X_H$ and $X_0(5)$ where $H$ is in the set \\ \{9.18.0.1,9.18.0.2\}. In both of these cases, there is a map from $C$ to the fiber product $C'$ of $X_{3.6.0.1}$. We will show that any non cuspidal rational point $E$ on $C$ cannot have $\rho_{E,3^{\infty}}(\G_{\Q})$ conjugate to $H.$ The modular curve $C'$ is an elliptic curve with Weierstrass model \[y^2+(x+1)y=x^3+x^2-5x+2\] that has $8$ rational points. There are four rational cusps and remaining four points correspond to two non CM elliptic curves \href{https://lmfdb.org/EllipticCurve/Q/338/b/2}{338.b2} and \href{https://lmfdb.org/EllipticCurve/Q/338/b/1}{338.b1} that have $j$-invariants $\frac{1331}{8}$ and $\frac{-1680914269}{32768}$ respectively. If there is a non cuspidal point on $C$ then it lies above either of these $j$-invariants so it must be quadratic twist of 338.b2 or 338.b1 but for any quadratic twist of these curves $\rho_{E,3^{\infty}}(\G_{\Q})$ must be conjugate to 3.6.0.1 because 3.6.0.1 has no proper subgroups $H'$ of index $2$ such that $\pm H'=3.6.0.1.$
\end{proof}

\begin{proposition}
    Let $E/\Q$ be without CM. If the isogeny-torsion graph associated to $E$ is of type $L_2(5)$, then one of the following holds.
    \begin{enumerate}
        \item If some vertex has a torsion point of order 5, then $\rho_{E,3^{\infty}}(\G_{\Q})$ is $\GL_2(\ZZ_3).$

        \item If all vertices have trivial torsion, then either $\rho_{E,3^{\infty}}(\G_{\Q})$ is $\GL_2(\ZZ_3)$ or it lies in the set \{3.3.0.1, 3.6.0.1, 9.9.0.1\}.
    \end{enumerate} 
\end{proposition}

\begin{proof}
\noindent
\begin{enumerate}
    \item If some vertex has torsion point of order 5, then we know from \cref{lemma:3.3.0.1}, \cref{cor:3.3.0.1}, \cref{cor:9.18.0.2}, \cref{lemma:9.27.0.1} and \cref{lemma:9.27.0.2} that the only possibility for $\rho_{E,3^{\infty}}(\G_{\Q})$ is $\GL_2(\ZZ_3).$

    \item If all vertices have trivial torsion, then from \cref{lemma:no5isogeny} and \cref{lemma:9.27.0.1} we know that either $\rho_{E,3^{\infty}}(\G_{\Q})$ is $\GL_2(\ZZ_3)$ or it is in \{3.3.0.1, 3.6.0.1, 9.9.0.1\}. All of these possibilities occur. See, \href{https://lmfdb.org/EllipticCurve/Q/1369/a/}{1369.a}, \href{https://lmfdb.org/EllipticCurve/Q/338/b/}{338.b} and \href{https://lmfdb.org/EllipticCurve/Q/43264/f/}{43264.f}.
\end{enumerate}
     
\end{proof}

\begin{proposition}
Let $E/\Q$ be without CM.
    \begin{enumerate}
        \item  If the isogeny-torsion graph associated to $E$ is of type $L_2(7)$, then $\rho_{E,3^{\infty}}(\G_{\Q})$ is $\GL_2(\ZZ_3).$ 

        \item If the isogeny-torsion graph associated to $E$ is of type $L_2(11)$, then $\rho_{E,3^{\infty}}(\G_{\Q})$ is $\GL_2(\ZZ_3).$
    \end{enumerate}
\end{proposition}

\begin{proof}
\noindent
    \begin{enumerate}
        \item From \cref{lemma:3.3.0.1}, \cref{cor:3.3.0.1}, \cref{cor:9.18.0.2}, \cref{lemma:9.27.0.1} and \cref{lemma:9.27.0.2} we know that if $E$ admits a degree 7 isogeny, then the only possibility for $\rho_{E,3^{\infty}}(\G_{\Q})$ is $\GL_2(\ZZ_3).$ 

        \item The genus of $X_0(11)$ is $1$ and it has analytic rank $0.$ Its Weierstrass model is given by the equation $y^2+y=x^{3}-x^{2}-10x-20.$ By using Nagell-Lutz Theorem we compute that there are $5$ rational points on $X_0(11)$, two are cusps, one is a rational CM point and two of them correspond to non CM elliptic curves, \href{https://lmfdb.org/EllipticCurve/Q/121/a/1}{121.a1} and \href{https://lmfdb.org/EllipticCurve/Q/121/a/2}{121.a2}, both of them have maximal 3-adic image. Since $E$ must be a quadratic twist of 121.a1 or 121.a2, it follows that $\rho_{E,3^{\infty}}(\G_{\Q})$ is $\GL_2(\ZZ_3).$
\end{enumerate}
\end{proof}

\begin{lemma}\label{lemma:13isogeny}
    Let $E/\Q$ be without CM. If $\rho_{E,3^{\infty}}(\G_{\Q})$ is conjugate to a subgroup of 3.3.0.1 or 9.27.0.1, then $E$ cannot admit a degree 13 isogeny.
\end{lemma}

\begin{proof}
    The fiber product of $X_{3.3.0.1}$ and $X_0(13)$ is of genus 2, is hyperelliptic of rank $0.$ Using command \texttt{Chabauty0} we compute all its rational points and both are cusps. The fiber product in other case is of genus 26, its canonical model is cut out by $276$ quadrics in $\PP^{25}$ and has no points modulo 5. This computation can be verified using script \texttt{9.27.0.1} of \cite{RakviGitHub}.
\end{proof}

\begin{proposition}
Let $E/\Q$ be without CM.
\begin{enumerate}
     \item If the isogeny-torsion graph associated to $E$ is of type $L_2(13)$, then $\rho_{E,3^{\infty}}(\G_{\Q})$ is $\GL_2(\ZZ_3).$ 

     \item If the isogeny-torsion graph associated to $E$ is of type $L_2(17)$, then $\rho_{E,3^{\infty}}(\G_{\Q})$ is $\GL_2(\ZZ_3).$ 
     
     \item If the isogeny-torsion graph associated to $E$ is of type $L_2(37)$, then $\rho_{E,3^{\infty}}(\G_{\Q})$ is $\GL_2(\ZZ_3).$ 
     \end{enumerate}
\end{proposition}

\begin{proof}
\noindent
\begin{enumerate}
    \item From \cref{lemma:13isogeny}, the only possibility for $\rho_{E,3^{\infty}}(\G_{\Q})$ is $\GL_2(\ZZ_3)$ because any subgroup from Theorem \ref{thm:no3isogeny} is contained in either 3.3.0.1 or 9.27.0.1.
    
    \item The genus of $X_0(17)$ is $1$ and it has rank $0.$ Its Weierstrass model is given by the equation $y^2+(x+1)y=x^{3} - x^{2}-x-14.$ There are $4$ rational points on $X_0(17)$, two are rational cusps and two of them correspond to non CM elliptic curves, \href{https://lmfdb.org/EllipticCurve/Q/14450/b/1}{14450.b1} and \href{https://lmfdb.org/EllipticCurve/Q/14450/b/2}{14450.b2}. Both of them have maximal 3-adic image. Since $E$ must be a quadratic twist of 14450.b1 or 14450.b2, it follows that $\rho_{E,3^{\infty}}(\G_{\Q})$ is $\GL_2(\ZZ_3).$

      \item We know that there are only two non cuspidal points on $X_0(37)$ \cite{MR482230} and both of them correspond to non CM elliptic curves with labels \href{https://lmfdb.org/EllipticCurve/Q/1225/b/1}{1225.b1} and \href{https://lmfdb.org/EllipticCurve/Q/1225/b/2}{1225.b2} which have maximal 3-adic image. Since $E$ must be a quadratic twist of 1225.b1 or 1225.b2, it follows that $\rho_{E,3^{\infty}}(\G_{\Q})$ is $\GL_2(\ZZ_3).$
      \end{enumerate}
\end{proof}

We will now discuss linear graphs of type $L_3(p^2)$ for $p \in \{3,5\}$, i.e.,
\begin{center}

\begin{tikzcd}[ampersand replacement=\&]
E_1 \arrow[r, "p", no head] \& E_2 \arrow[r, "p", no head] \& E_3
\end{tikzcd}
\end{center}

\begin{lemma}\label{lemma:3.3.0.1xX0(25)}
    Let $E/\Q$ be without CM such that $\rho_{E,3^{\infty}}(\G_{\Q})$ is conjugate to subgroup of 3.3.0.1. Then, $E$ cannot admit a cyclic isogeny of degree 25.
\end{lemma}

\begin{proof}
    Assume that there exists such an elliptic curve. Then, there exist elliptic curves $E'$ and $E''$ which are isogenous to $E$ arranged like $E \to E' \to E''$ with a degree 5 isogeny between $E,E'$ and between $E',E''.$ Here $E'$ has two independent degree 5 isogenies, therefore $\rho_{E',5}(\G_{\Q})$ is contained in the split cartan subgroup of $\GL_2(\ZZ/5\ZZ).$ Hence, it suffices to show that there are no non cuspidal non CM points on the fiber product $C$ of $X_{3.3.0.1}$ and $X_{sp}(5).$ It has LMFDB label \href{https://beta.lmfdb.org/ModularCurve/Q/15.90.4.a.1/}{15.90.4.1}. Its Jacobian has rank $1$. The canonical model for $C$ is given by equations\[6x^2+xy+4y^2-xz+2yz-z^2=0,
         x^2y+xy^2-y^3+4xyz+2y^2z+4yz^2+w^3=0.\] The Jacobian of $C$ has rank 1. There are two obvious rational points on $C$, when $y=w=0$, we get $P_1:=(1/2,0,1,0)$ and $P_2:=(-1/3,0,1,0).$

We find an automorphism $i \colon C(\Q) \to C(\Q)$ given by \begin{align*}
    [x,y,z,w] \to & \bigg[\frac{3}{11}y^4 + \frac{7}{11}y^3z + \frac{20}{11}y^2z^2 + \frac{25}{11}yz^3 + \frac{1}{55}yw^3 + \frac{7}{11}zw^3, \\
&  y^4 - \frac{6}{11}yw^3, \\
& \frac{-3}{11}y^4 + \frac{4}{11}y^3z - \frac{20}{11}y^2z^2 - \frac{25}{11}yz^3 - \frac{1}{55}yw^3 - \frac{13}{11}zw^3,\\
& y^3w - \frac{6}{11}w^4\bigg]
\end{align*}
 defined at points when at least one of $y$ or $w$ is non-zero, it can be extended to $P_1$, $P_2$ as well.
 We will now show that there are no fixed points of $i.$ Assume there is a fixed point, i.e, $i(P)=P$. From \[x^2y+xy^2-y^3+4xyz+2y^2z+4yz^2+w^3=0\] we know that if $y=0$, then $w=0.$ So, assume that $y \ne 0$, if $y^4-\frac{6}{11}yw^3=y$, then since $y \ne 0$, we can take $y=1$ so $1-\frac{6}{11}w^3=1$ so $w=0.$ If $w=0$ and $y=1$, then plugging these values in equation \[\frac{-3}{11}y^4 + \frac{4}{11}y^3z - \frac{20}{11}y^2z^2 - \frac{25}{11}yz^3 - \frac{1}{55}yw^3 - \frac{13}{11}zw^3=z\] we get \[\frac{-3}{11} + \frac{4}{11}z - \frac{20}{11}z^2 - \frac{25}{11}z^3 =z\] which has no rational solutions.
 
 The quotient of $C$ by this automorphism is an elliptic curve $E$ of rank $1.$ So, there exists an abelian variety, say $V$ such that the Jacobian of $C$, denoted by $J_C$ decomposes into $E \times V.$ The rank of $E \times V$ is sum of rank of $E$ and rank of $V$ which is equal to rank of $J_C$. Since $E$ has rank $1$ and $J_C$ also has rank $1$, the rank of $V$ is $0$. For any point $P \in C(\Q)$, the point $P-i(P)$ lies in $J_C(\Q)_{\tors}.$ Hence, it suffices to compute preimages of rational points in $J_C(\Q)_{\tors}$ under the map $a \colon C(\Q) \to J_C(\Q)_{\tors}$ given by $P \mapsto P-i(P)$ which is injective away from the fixed points of $i.$ We were not able to compute $J_C(\Q)_{\tors}$ but using local computations which we describe below we were able to deduce that $J_C(\Q)_{\tors}$ is a subgroup of $\ZZ/2\ZZ \times \ZZ/6\ZZ$. Let $p$ be an odd prime not equal to $3$ or $5$. From the Lemma in appendix of \cite{MR604840}, we know that $J_C(\Q)_{\tors}$ injects into $J_C(\ZZ/p\ZZ).$ Using \texttt{ClassGroup} command of \texttt{MAGMA} for primes $7$ and $13$ we find that $J_C(\Q)_{\tors}$ must be a subgroup of $\ZZ/2\ZZ \times \ZZ/6\ZZ.$ We checked that there is no point $P \in C(\Q)$ such that $P-i(P) \in J_C(\Q)_{\tors} \injects J_C(\ZZ/7\ZZ)$ is of order 2 or of order 6 as follows. We checked that for any two points $P,Q$ in reduction of $C$ modulo $7$, $P-Q$ is not of order 2 or order 6. There are two points in $C(\Q)$ that are easy to observe, $P_1=(1/2,0,1,0)$ and $P_2=(-1/3,0,1,0)=i(P_1)$. Working modulo $7$, we verify that the divisor $D \colon P_1-P_2$ is of order 3. Further, we observe that $a(P_2)=2D$, $a(P_1)=D$ and there are no fixed points of $i.$ So, we have computed all the rational points of $C$ and finally we verify that $P_1$ and $P_2$ are cusps. So, we are done. These computations can be verified using script \texttt{3.3.0.1 x X0(25)}. 
\end{proof}

\begin{proposition}
    Let $E/\Q$ be without CM. If the isogeny-torsion graph associated to $E$ is of type $L_3(25)$, then   $\rho_{E,3^{\infty}}(\G_{\Q})$ is $\GL_2(\ZZ_3).$
\end{proposition}

\begin{proof}
    It suffices to show that $\rho_{E,3^{\infty}}(\G_{\Q})$ cannot be conjugate to 3.3.0.1 or 9.27.0.1 because any subgroup from Theorem \ref{thm:no3isogeny} is contained in either 3.3.0.1 or 9.27.0.1. From Lemma \ref{lemma:3.3.0.1xX0(25)} and Lemma \ref{lemma:9.27.0.1} we know both of these groups cannot admit a cyclic isogeny of degree 25. Hence, proved.
\end{proof}

\begin{proposition}
    Let $E/\Q$ be without CM. If the isogeny-torsion graph associated to $E$ is of type $L_3(9)$, then one of the following holds. 
    \begin{enumerate}
        \item If all vertices have trivial torsion, then $(\rho_{E_i,3^{\infty}}(\G_{\Q}))_{i=1}^{i=3}$ is given by one of the three tuples (9.12.0.1,3.12.0.1,9.12.0.1), (9.36.0.5,9.36.0.1,9.36.0.4), (9.36.0.6,9.36.0.3,9.36.0.6),\\ (27.36.0.1,9.36.0.2,27.36.0.1).

        \item If two vertices have cyclic torsion of order 3, then of $(\rho_{E_i,3^{\infty}}(\G_{\Q}))_{i=1}^{i=3}$ is given by one of the three tuples (9.24.0.3,3.24.0.1,9.24.0.1) , (9.72.0.11,9.72.0.2,9.72.0.6), (9.72.0.13,9.72.0.4,9.72.0.7), (27.72.0.2,9.72.0.3,27.72.0.1).

        \item Otherwise, $(\rho_{E_i,3^{\infty}}(\G_{\Q}))_{i=1}^{i=3}$ is given by (9.72.0.12,9.72.0.1,9.72.0.5).
    \end{enumerate} 
    
\end{proposition}

\begin{proof} Since $E_1$ and $E_3$ admits a cyclic isogeny of degree 9, their 3-adic image is in the set \{9.12.0.1, 9.24.0.1, 9.24.0.3, 9.36.0.4, 9.72.0.5, 9.72.0.11, 9.36.0.5, 9.72.0.6, 9.72.0.12, 9.36.0.6, 9.72.0.7, 9.72.0.13, 27.36.0.1, 27.72.0.1, 27.72.0.2\} because these are the only subgroups of 9.12.0.1. 
    \begin{enumerate}
        \item Assume all the vertices have trivial torsion. Then, $\rho_{E_1,3^{\infty}}(\G_{\Q})$ cannot be in the set \{9.24.0.1, 9.24.0.3, 9.72.0.5, 9.72.0.11, 9.72.0.6, 9.72.0.12, 9.72.0.7, 9.72.0.13, 27.72.0.1, 27.72.0.2\} because then $\rho_{E_2,3^{\infty}}(\G_{\Q})$ would have a point of order 3 in its torsion group from Table \ref{table:3adicimages} and Lemma \ref{lemma:torsion3}. For remaining possibilities of $\rho_{E_1,3^{\infty}}(\G_{\Q})$ we look at Table \ref{table:3adicimages} to get $(\rho_{E_i,3^{\infty}}(\G_{\Q}))_{i=1}^{i=3}$. As examples see \href{https://lmfdb.org/EllipticCurve/Q/175/b/}{175.b}, \href{https://lmfdb.org/EllipticCurve/Q/432/b/}{432.b}, \href{https://lmfdb.org/EllipticCurve/Q/22491/u/}{22491.u} and \href{https://lmfdb.org/EllipticCurve/Q/304/c/}{304.c} of arrangements \\(9.12.0.1,3.12.0.1,9.12.0.1 , (9.36.0.5,9.36.0.1,9.36.0.4), (9.36.0.6,9.36.0.3,9.36.0.6) and \\(27.36.0.1,9.36.0.2,27.36.0.1) respectively.

        \item If two vertices have cyclic torsion of order 3, then $\rho_{E_3,3^{\infty}}(\G_{\Q})$ lies in \{9.24.0.1, 9.72.0.5, 9.72.0.6, 9.72.0.7, 27.72.0.1\}. From Lemma \ref{lemma:torsion9} we know that it cannot be 9.72.0.5. So from Table \ref{table:3adicimages} we get that $(\rho_{E_i,3^{\infty}}(\G_{\Q}))_{i=1}^{i=3}$ is given by one of the four tuples (9.24.0.3,3.24.0.1,9.24.0.1) , (9.72.0.11,9.72.0.2,9.72.0.6), (9.72.0.13,9.72.0.4,9.72.0.7), (27.72.0.2,9.72.0.3,27.72.0.1). See examples, \href{https://lmfdb.org/EllipticCurve/Q/26/a/}{26.a}, \href{https://lmfdb.org/EllipticCurve/Q/54/a/}{54.a}, \href{https://lmfdb.org/EllipticCurve/Q/3213/k/}{3213.k} and \href{https://lmfdb.org/EllipticCurve/Q/19/a/}{19.a} respectively.

        \item If $E_3(\Q)_{\tors}$ is cyclic of order 9, then from Lemma \ref{lemma:torsion9} the only possibility for $\rho_{E_3,3^{\infty}}(\G_{\Q})$ is 9.72.0.5. Therefore, from Table \ref{table:3adicimages} $(\rho_{E_i,3^{\infty}}(\G_{\Q}))_{i=1}^{i=3}$ is given by (9.72.0.12,9.72.0.1,9.72.0.5). See \href{https://lmfdb.org/EllipticCurve/Q/54/b/}{54.b}.
    \end{enumerate}
\end{proof}

\subsection{Rectangular graphs}

From Table $3$ of \cite{chiloyan2021classification} we know that there are two types of rectangular isogeny graphs, \begin{center}

\begin{tikzcd}[ampersand replacement=\&]
E_1 \arrow[r, "q", no head] \arrow[d, "p", no head] \& E_2 \arrow[d, "p"', no head] \\
E_3 \arrow[r, "q", no head]                         \& E_4                         
\end{tikzcd}
\end{center} denoted by $R_4(pq)$ and \begin{center}
    
\begin{tikzcd}[ampersand replacement=\&]
E_1 \arrow[r, "3", no head] \arrow[d, "2", no head] \& E_3 \arrow[d, "2"', no head] \arrow[r, "3", no head] \& E_5 \arrow[d, "2", no head] \\
E_2 \arrow[r, "3"', no head]                        \& E_4 \arrow[r, "3"', no head]                         \& E_6                        
\end{tikzcd}
\end{center} denoted by $R_6.$ 

\begin{proposition}
Let $E/\Q$ be without CM. If the isogeny graph associated to $E$ is of type $R_4(6)$, i.e., $(p,q)=(3,2)$ and 

\begin{enumerate}
    \item If all vertices have torsion subgroup defined over $\Q$ isomorphic to $\ZZ/2\ZZ$, then $(\rho_{E_i,3^{\infty}}(\G_{\Q}))_{i=1}^{i=4}$ is given by the four tuple (3.4.0.1, 3.4.0.1, 3.4.0.1, 3.4.0.1).

    \item If some vertex has torsion subgroup defined over $\Q$ isomorphic to $\ZZ/6\ZZ$, then $(\rho_{E_i,3^{\infty}}(\G_{\Q}))_{i=1}^{i=4}$ is given by the four tuple (3.8.0.2,3.8.0.2,3.8.0.1,3.8.0.1).

\end{enumerate} 
\end{proposition}

\begin{proof}
\noindent
\begin{enumerate}
    \item Assume that all vertices have torsion subgroup isomorphic to $\ZZ/2\ZZ.$ Then, from Lemma \ref{lemma:9.12.0.2tors2}, Corollary \ref{lemma:9.36.0.7-9.360.9order2}, Lemma \ref{lemma:9.36.0.1-9.36.0.627.36.0.1order2} and Lemma \ref{lemma:torsion3} we know that $\rho_{E_i,3^{\infty}}(\G_{\Q})$ for $i \in \{1,2,3,4\}$ lies in \{3.4.0.1,3.8.0.2,3.12.0.1,9.12.0.1,9.24.0.3\}. If $\rho_{E_1,3^{\infty}}(\G_{\Q})$ is 3.8.0.2, then from Table \ref{table:3adicimages} $\rho_{E_3,3^{\infty}}(\G_{\Q})$ is 3.8.0.1 which is not possible. If $\rho_{E_1,3^{\infty}}(\G_{\Q})$ is 9.24.0.3, then from Table \ref{table:3adicimages} we know that $\rho_{E_3,3^{\infty}}(\G_{\Q})$ is 3.24.0.1 which is also not possible. If $\rho_{E_1,3^{\infty}}(\G_{\Q})$ is 9.12.0.1, then $E_1$ admits a cyclic isogeny of degree 9 however maximal degree of cyclic isogeny that $E_i$ can possess is degree 6 for $i \in \{1,2,3,4\}.$ If $\rho_{E_1,3^{\infty}}(\G_{\Q})$ is 3.12.0.1, then from Table \ref{table:3adicimages} we know that $\rho_{E_3,3^{\infty}}(\G_{\Q})$ is 9.12.0.1 which is a contradiction. Therefore, $\rho_{E_1,3^{\infty}}(\G_{\Q})$ is 3.4.0.1. From Table \ref{table:3adicimages} $\rho_{E_i,3^{\infty}}(\G_{\Q})$ is 3.4.0.1 for all $i.$ For example of such an isogeny class, see \href{https://lmfdb.org/EllipticCurve/Q/80/b/}{80.b}.

    \item Without loss of generality, we can assume that $(E_i(\Q)_{\tors})_{i=1}^{i=4}$ is $(\ZZ/2\ZZ,\ZZ/2\ZZ,\ZZ/6\ZZ,\ZZ/6\ZZ).$ From argument used in previous part, we know that $\rho_{E_1,3^{\infty}}(\G_{\Q})$ lies in \{3.4.0.1,3.8.0.2\}. It cannot be 3.4.0.1 because from Table \ref{table:3adicimages} we know that $\rho_{E_i,3^{\infty}}(\G_{\Q})$ for $i \in \{2,3,4\}$ will also be 3.4.0.1 which is a contradiction because from Lemma \ref{lemma:torsion3}, $E_3(\Q)_{\tors}$ cannot contain a point of order 3. Therefore, $(\rho_{E_i,3^{\infty}}(\G_{\Q}))_{i=1}^{i=4}$ is given by the four tuple (3.8.0.2,3.8.0.2,3.8.0.1,3.8.0.1). See, \href{https://lmfdb.org/EllipticCurve/Q/20/a/}{20.a}.
    
    \end{enumerate}

\end{proof}

\begin{proposition}
Let $E/\Q$ be without CM. If the isogeny graph associated to $E$ is of type $R_4(10)$ ,i.e., $(p,q)=(2,5)$ then $\rho_{E,3^{\infty}}(\G_{\Q})$ for all vertices is $\GL_2(\ZZ_3).$
\end{proposition}

\begin{proof}
    From Table 3 of \cite{chiloyan2021classification} we know that there exists at least one vertex $E$ such that $E(\Q)_{\tors}$ is isomorphic to $\ZZ/2\ZZ$, then using Lemma \ref{lemma:9.9.0.1}, Lemma \ref{lemma:9.18.0.1}, Corollary \ref{cor:9.18.0.2}, Lemma \ref{lemma:9.27.0.1} and Lemma \ref{lemma:9.27.0.2} we know that $\rho_{E,3^{\infty}}(\G_{\Q})$ lies in the set \{3.3.0.1,3.6.0.1\}. Since 3.6.0.1 is a subgroup of 3.3.0.1, it suffices to show that if $\rho_{E_,3^{\infty}}(\G_{\Q})$ is conjugate to a subgroup of 3.3.0.1, then $E$ cannot admit a cyclic isogeny of degree 10. Let $C$ be the fiber product of $X_{3.3.0.1}$ and $X_0(10).$ Then, $C$ has genus $2$, rank $0$ has Weierstrass model given by \[ y^{2} + (x^{3} + 1) y =4 x^{3}.\] There are four rational points in $C$ and all of them are cusps. Hence, the proof follows for $R_4(10).$ 
    
\end{proof}

\begin{proposition}\label{prop:R_4(15)}
Let $E/\Q$ be without CM. If the isogeny graph associated to $E$ is of type $R_4(15)$, i.e., $(p,q)=(3,5)$ then one of the following holds: \begin{enumerate}
    \item If none of the vertices have torsion group isomorphic to $\ZZ/3\ZZ$, then $(\rho_{E_i,3^{\infty}}(\G_{\Q}))_{i=1}^{i=4}$ is given by the four tuple (3.4.0.1,3.4.0.1,3.4.0.1,3.4.0.1).

    \item If $E_1$, $E_2$ have torsion subgroup isomorphic to $\ZZ/3\ZZ$ and $E_3$, $E_4$ have trivial torsion, then $(\rho_{E_i,3^{\infty}}(\G_{\Q}))_{i=1}^{i=4}$ is given by the four tuple \\(3.8.0.1,3.8.0.1,3.8.0.2,3.8.0.2).
\end{enumerate}
\end{proposition}

\begin{proof}
    From Proposition \ref{prop:X0(15)} we know that $\rho_{E_1,3^{\infty}}(\G_{\Q})$ lies in \{3.4.0.1, 3.8.0.1, 3.8.0.2.\} If $E_1(\Q)_{\tors}$ is isomorphic to $\ZZ/3\ZZ$, then from Lemma \ref{lemma:torsion3} and Table \ref{table:3adicimages} it follows that $(\rho_{E_i,3^{\infty}}(\G_{\Q}))_{i=1}^{i=4}$ is given by the four tuple (3.8.0.1,3.8.0.1,3.8.0.2,3.8.0.2). For example, see \href{https://lmfdb.org/EllipticCurve/Q/50/a/}{50.a}. Let us assume that none of the vertices have torsion group isomorphic to $\ZZ/3\ZZ$, in this case $\rho_{E_1,3^{\infty}}(\G_{\Q})$ cannot be 3.8.0.1, it follows from Lemma \ref{lemma:torsion3}. It cannot be 3.8.0.2 either because using Table \ref{table:3adicimages} it will follow that $\rho_{E_3,3^{\infty}}(\G_{\Q})$ is 3.8.0.1 which will be a contradiction from Lemma \ref{lemma:torsion3} because we know that none of the vertices have torsion group isomorphic to $\ZZ/3\ZZ$. Therefore, $\rho_{E_1,3^{\infty}}(\G_{\Q})$ is 3.4.0.1 and from Table \ref{table:3adicimages} we know that $(\rho_{E_i,3^{\infty}}(\G_{\Q}))_{i=1}^{i=4}$ is given by the four tuple (3.4.0.1,3.4.0.1,3.4.0.1,3.4.0.1). For example, see \href{https://lmfdb.org/EllipticCurve/Q/50/b/}{50.b}.
\end{proof}

\begin{proposition}
    Let $E/\Q$ be without CM. If the isogeny graph associated to $E$ is of type $R_4(21)$, i.e., $(p,q)=(3,7)$ then one of the following holds: 
    
    \begin{enumerate}
    \item If none of the vertices have torsion group isomorphic to $\ZZ/3\ZZ$, then $(\rho_{E_i,3^{\infty}}(\G_{\Q}))_{i=1}^{i=4}$ is given by the four tuple (3.4.0.1,3.4.0.1,3.4.0.1,3.4.0.1).

    \item If $E_1$, $E_2$ have torsion subgroup isomorphic to $\ZZ/3\ZZ$ and $E_3$, $E_4$ have trivial torsion, then $(\rho_{E_i,3^{\infty}}(\G_{\Q}))_{i=1}^{i=4}$ is given by the four tuple \\(3.8.0.1,3.8.0.1,3.8.0.2,3.8.0.2).
\end{enumerate}
\end{proposition}

\begin{proof}
    The modular curve $X_0(21)$ has genus $1$, rank $0$ and is given by the Weierstrass equation \[y^2+xy=x^3-4x-1.\] It has eight rational points, out of which four are cusps and remaining four points are non cuspidal non CM points which correspond to elliptic curves \href{https://www.lmfdb.org/EllipticCurve/Q/162/b/1}{162.b1}, \href{https://www.lmfdb.org/EllipticCurve/Q/162/b/2}{162.b2}, \href{https://www.lmfdb.org/EllipticCurve/Q/162/b/3}{162.b3} and \href{https://www.lmfdb.org/EllipticCurve/Q/162/b/4}{162.b4}. Therefore, if $E$ is an elliptic curve that admits a cyclic isogeny of degree 21, then $E$ must be a quadratic twist of these four curves. Hence, $\rho_{E,3^{\infty}}(\G_{\Q})$ must lie in \{3.4.0.1, 3.8.0.1, 3.8.0.2\}. Proof follows from similar arguments as given in Proposition \ref{prop:R_4(15)}. See examples \href{https://lmfdb.org/EllipticCurve/Q/1296/f/}{1296.f} and \href{https://lmfdb.org/EllipticCurve/Q/162/b/}{162.b} respectively.
\end{proof}

\begin{proposition}
    Let $E/\Q$ be without CM. If the isogeny graph associated to $E$ is of type $R_6$ then one of the following holds: \begin{enumerate}
        \item If all vertices have torsion group isomorphic to $\ZZ/2\ZZ$, then $(\rho_{E_i,3^{\infty}}(\G_{\Q}))_{i=1}^{i=6}$ is given by the six tuple (9.12.0.1, 9.12.0.1, 3.12.0.1, 3.12.0.1, 9.12.0.1, 9.12.0.1).

        \item If $(E_i(\Q)_{\tors})_{i=1}^{i=6}$ is $(\ZZ/2\ZZ,\ZZ/2\ZZ,\ZZ/6\ZZ,\ZZ/6\ZZ,\ZZ/6\ZZ,\ZZ/6\ZZ)$, then $(\rho_{E_i,3^{\infty}}(\G_{\Q}))_{i=1}^{i=6}$ is given by the six tuple (9.24.0.3, 9.24.0.3, 3.24.0.1, 3.24.0.1, 9.24.0.1, 9.24.0.1).
    \end{enumerate}
\end{proposition}

\begin{proof}
    Since $E_1$ admits a cyclic isogeny of degree 9, then $\rho_{E_1,3^{\infty}}(\G_{\Q})$ must be a subgroup of 9.12.0.1. From Corollary \ref{lemma:9.36.0.7-9.360.9order2} and Lemma \ref{lemma:9.36.0.1-9.36.0.627.36.0.1order2} we know that $\rho_{E_1,3^{\infty}}(\G_{\Q})$ lies in \{9.12.0.1,9.24.0.1,9.24.0.3\}. If all vertices have torsion group isomorphic to $\ZZ/2\ZZ$, then $\rho_{E_1,3^{\infty}}(\G_{\Q})$ cannot be 9.24.0.1 from Lemma \ref{lemma:torsion3} or 9.24.0.3 because from Table \ref{table:3adicimages} we know that $\rho_{E_3,3^{\infty}}(\G_{\Q})$ will be 3.24.0.1 which is also a contradiction from Lemma \ref{lemma:torsion3}. Therefore, $\rho_{E_1,3^{\infty}}(\G_{\Q})$ is 9.12.0.1. Using Table \ref{table:3adicimages} we conclude that $(\rho_{E_i,3^{\infty}}(\G_{\Q}))_{i=1}^{i=6}$ is given by the six tuple (9.12.0.1, 9.12.0.1, 3.12.0.1, 3.12.0.1, 9.12.0.1, 9.12.0.1). For example, see \href{https://lmfdb.org/EllipticCurve/Q/98/a/}{98.a}.

    If $(E_i(\Q)_{\tors})_{i=1}^{i=6}$ is $(\ZZ/2\ZZ,\ZZ/2\ZZ,\ZZ/6\ZZ,\ZZ/6\ZZ,\ZZ/6\ZZ,\ZZ/6\ZZ)$, then from Lemma \ref{lemma:torsion3} we know that $\rho_{E_1,3^{\infty}}(\G_{\Q})$ is either 9.12.0.1 or 9.24.0.3, if it is 9.12.0.1 then $\rho_{E_3,3^{\infty}}(\G_{\Q})$ is 3.12.0.1 using Table \ref{table:3adicimages} but then from Lemma \ref{lemma:torsion3} we know that $E_3(\Q)_{\tors}$ cannot be $\ZZ/6\ZZ$ which is a contradiction. Hence, $\rho_{E_1,3^{\infty}}(\G_{\Q})$ is 9.24.0.3. Using Table \ref{table:3adicimages} we conclude that $(\rho_{E_i,3^{\infty}}(\G_{\Q}))_{i=1}^{i=6}$ is given by the six tuple (9.24.0.3, 9.24.0.3, 3.24.0.1, 3.24.0.1, 9.24.0.1, 9.24.0.1). For example, see \href{https://lmfdb.org/EllipticCurve/Q/14/a/}{14.a}.
\end{proof}

\subsection{Isogeny graphs of type $T_k$}

From Table $3$ of \cite{chiloyan2021classification} we know that there are three type $T_k$ graphs

$T_4$ 
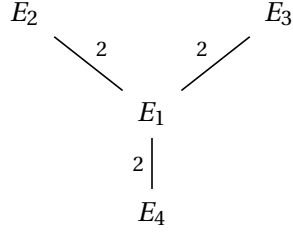
\begin{center}\begin{tikzcd}[ampersand replacement=\&]
E_2 \&                                                                                 \& E_3 \\
    \& E_1 \arrow[ru, "2", no head] \arrow[lu, "2"', no head] \arrow[d, "2"', no head] \&     \\
    \& E_4                                                                             \&    
\end{tikzcd} \end{center}

$T_6$ 
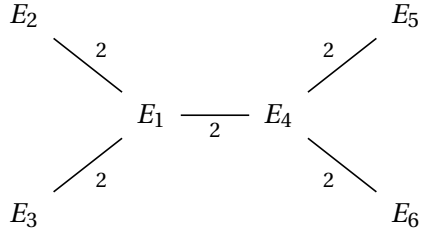
\begin{center}\begin{tikzcd}[ampersand replacement=\&]
E_2 \&                                                                                 \&                                                        \& E_5 \\
    \& E_1 \arrow[lu, "2"', no head] \arrow[ld, "2", no head] \arrow[r, "2"', no head] \& E_4 \arrow[ru, "2", no head] \arrow[rd, "2"', no head] \&     \\
E_3 \&                                                                                 \&                                                        \& E_6
\end{tikzcd} 
\end{center}

$T_8$  
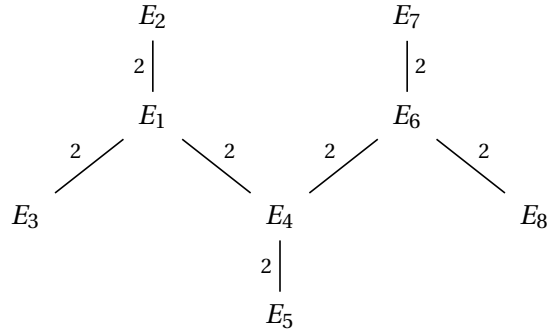
\begin{center}
        \begin{tikzcd}[ampersand replacement=\&]
    \& E_2                                                                            \&                              \& E_7                                                                             \&     \\
    \& E_1 \arrow[u, "2", no head] \arrow[ld, "2"', no head] \arrow[rd, "2", no head] \&                              \& E_6 \arrow[ld, "2"', no head] \arrow[rd, "2", no head] \arrow[u, "2"', no head] \&     \\
E_3 \&                                                                                \& E_4 \arrow[d, "2"', no head] \&                                                                                 \& E_8 \\
    \&                                                                                \& E_5                          \&                                                                                 \&    
\end{tikzcd}
\end{center}

\begin{proposition}
Let $E/\Q$ be without CM. If the isogeny graph associated to $E$ is of type $T_k$ for $k \in \{4,6,8\}$, then $\rho_{E,3^{\infty}}(\G_{\Q})$ at every vertex is $\GL_2(\ZZ_3).$
\end{proposition}

\begin{proof}
    Since all the edges have label 2, it suffices to determine 3-adic image at any one vertex. From \cite{chiloyan2021classification}, we know that if the isogeny graph associated to $E$ is of type $T_k$ for $k \in \{4,6,8\}$, then at least one of the vertices should have torsion subgroup of order 4. From \cref{lemma:3.3.0.1} and \cref{lemma:9.27.0.1} we know that the only possibility for $\rho_{E,3^{\infty}}(\G_{\Q})$ is $\GL_2(\ZZ_3).$
\end{proof}

\subsection{Isogeny graphs of type $S$}
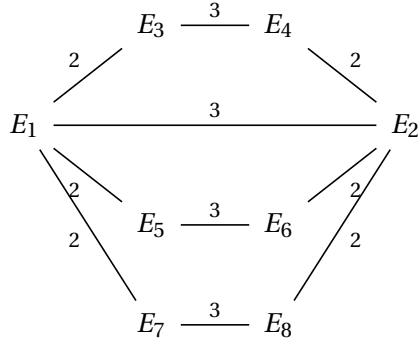
\begin{center}
    
    \begin{tikzcd}[ampersand replacement=\&]
                                                                                                            \& E_3 \arrow[r, "3", no head] \& E_4 \arrow[rd, "2", no head] \&                                                        \\
E_1 \arrow[rrr, "3", no head] \arrow[ru, "2", no head] \arrow[rd, "2"', no head] \arrow[rdd, "2"', no head] \&                             \&                              \& E_2 \arrow[ld, "2", no head] \arrow[ldd, "2", no head] \\
                                                                                                            \& E_5 \arrow[r, "3", no head] \& E_6                          \&                                                        \\
                                                                                                            \& E_7 \arrow[r, "3", no head] \& E_8                          \&                                                       
\end{tikzcd}
\end{center}

\begin{proposition}
    Let $E/\Q$ be without CM. If the isogeny graph associated to $E$ is of type $S$, then exactly one of the 
    following is true: 
    \begin{enumerate}
        \item If $E_1(\Q)_{\tors}$ is isomorphic to $\ZZ/2\ZZ \times \ZZ/2\ZZ$, then $\rho_{E,3^{\infty}}(\G_{\Q})$ of every vertex is 3.4.0.1.

        \item If $E_1(\Q)_{\tors}$ is isomorphic to $\ZZ/2\ZZ \times \ZZ/6\ZZ$, then $\rho_{E_k,3^{\infty}}(\G_{\Q})$ for odd $k$ is 3.8.0.1 and $\rho_{E_k,3^{\infty}}(\G_{\Q})$ for even $k$ is 3.8.0.2.
    \end{enumerate}
\end{proposition}

\begin{proof}

 Observe that every vertex in an isogeny graph of type $S$ admits a degree 3 isogeny. From \cref{thm:3isogeny}, we know the possibilities for $\rho_{E,3^{\infty}}(\G_{\Q})$ if $E$ admits a degree 3 isogeny. If $E(\Q)_{\tors}$ contains a subgroup of order 4, then from \cref{lemma:3.12.0.1}, \cref{lemma:9.12.0.1order4}, \cref{lemma:9.12.0.2tors2}, \cref{lemma:9.36.0.1-9.36.0.627.36.0.1order2} and \cref{lemma:9.36.0.7-9.360.9order2} we know that $\rho_{E,3^{\infty}}(\G_{\Q})$ is in the set \{3.4.0.1, 3.8.0.1 or 3.8.0.2.\}

    \begin{enumerate}
        \item If $E_1(\Q)_{\tors}$ is isomorphic to $\ZZ/2\ZZ \times \ZZ/2\ZZ$ and if $\rho_{E_1,3^{\infty}}(\G_{\Q})$ is 3.8.0.2, then from \cref{table:3adicimages}, $\rho_{E_2,3^{\infty}}(\G_{\Q})$ is 3.8.0.1 which is not possible because no vertex can admit a torsion point of order 3 (by looking at Table $4$ of \cite{chiloyan2021classification}) and from \cref{lemma:torsion3} we know that if $\rho_{E,3^{\infty}}(\G_{\Q})$ is 3.8.0.1, then $E$ has a torsion point of order 3. From \cref{lemma:isomorphismofisogenis} and \cref{table:3adicimages} we know that if $\rho_{E_1,3^{\infty}}(\G_{\Q})$ is 3.4.0.1, then label of every vertex is 3.4.0.1. For example, see \href{https://www.lmfdb.org/EllipticCurve/Q/150/b/}{150.b}.

        \item  If $E_1(\Q)_{\tors}$ is isomorphic to $\ZZ/2\ZZ \times \ZZ/6\ZZ$, then $E_1(\Q)_{\tors}$ contains a point of order 3 and a subgroup of order 4 so $\rho_{E_1,3^{\infty}}(\G_{\Q})$ is 3.8.0.1. From \cref{lemma:isomorphismofisogenis} and \cref{table:3adicimages} we know that if $\rho_{E_1,3^{\infty}}(\G_{\Q})$ is 3.8.0.1, then label of $\rho_{E_k,3^{\infty}}(\G_{\Q})$ for odd $k$ is 3.8.0.1 and $\rho_{E_k,3^{\infty}}(\G_{\Q})$ for even $k$ is 3.8.0.2. For example, see \href{https://www.lmfdb.org/EllipticCurve/Q/30/a/}{30.a}. 
    \end{enumerate}
\end{proof}

\section*{Acknowledgments}
I would like to thank Eran Assaf, Garen Chiloyan, \'Alvaro Lozano-Robledo, Jeremy Rouse, Andrew Sutherland and David Zureick-Brown for their useful suggestions and several helpful discussions. I would also like to thank anonymous referees for their useful and valuable comments.

\end{document}